\documentclass[11pt,a4paper,reqno]{amsart}

\usepackage{mathrsfs} 

\usepackage{amsmath,amssymb,graphics,epsfig,color,enumerate,psfrag,bbm}
\usepackage{dsfont}  
\usepackage{verbatim}
\usepackage{geometry}
\usepackage{hyperref}
\newtheorem{Theorem}{Theorem}
\newtheorem*{Theorem*}{Theorem}
\newtheorem{Lemma}[]{Lemma}
\newtheorem{Proposition}[]{Proposition}
\newtheorem{Corollary}[]{Corollary}
\theoremstyle{definition}
\newtheorem*{Remark}{Remark}

\newtheorem{Definition}[Theorem]{Definition}
\newtheorem*{Facts}{Facts}

\newtheorem{Assumption}{Assumption}
\newtheorem*{Notation}{Notation}

 \definecolor{darkgreen}{rgb}{0,0.4,0}

\definecolor{light}{gray}{.9}

\newcommand{\cB}{\ensuremath{\mathcal B}}

\newcommand{\cF}{\ensuremath{\mathcal F}}

\newcommand{\cH}{\ensuremath{\mathcal H}}

\newcommand{\cN}{\ensuremath{\mathcal N}}

\newcommand{\cW}{\ensuremath{\mathcal W}}
\newcommand{\cX}{\ensuremath{\mathcal X}}

\newcommand{\bbC}{{\ensuremath{\mathbb C}} }
\newcommand{\bbD}{{\ensuremath{\mathbb D}} }
\newcommand{\bbE}{{\ensuremath{\mathbb E}} }

\newcommand{\bbN}{{\ensuremath{\mathbb N}} }

\newcommand{\bbP}{{\ensuremath{\mathbb P}} }
\newcommand{\bbQ}{{\ensuremath{\mathbb Q}} }
\newcommand{\bbR}{{\ensuremath{\mathbb R}} }

\newcommand{\bbT}{{\ensuremath{\mathbb T}} }

\newcommand{\bbZ}{{\ensuremath{\mathbb Z}} }

\let\a=\alpha \let\b=\beta   \let\d=\delta  \let\e=\varepsilon
 \let\g=\gamma       
            
  \let\s=\sigma    \let\th=\vartheta

\newcommand{\ds}{\displaystyle}
\newcommand{\dd}{\mathrm{d}}

\newcommand{\de}{\partial}
\renewcommand\Re{\operatorname{Re}}
\renewcommand\Im{\operatorname{Im}}

\newcommand{\nt}{{\lfloor nt \rfloor}}
\newcommand{\ns}{{\lfloor  ns \rfloor}}

\author[V. Silvestri]{Vittoria Silvestri}\thanks{Statistical Laboratory, University of Cambridge. Research supported by EPSRC grant EP/H023348/1 for the Cambridge Centre for Analysis}

\address{Vittoria Silvestri. 
Statslab, Centre for Mathematical Sciences,
Wilberforce Road,
Cambridge,
CB3 0WA,
United Kingdom.}
\email{V.Silvestri@maths.cam.ac.uk}

\title[Fluctuation results for Hastings-Levitov planar growth]{Fluctuation results for Hastings-Levitov planar growth} 

\begin{document}
\maketitle

\begin{abstract} We study the fluctuations of the outer domain of Hastings-Levitov clusters in the small particle limit. These are shown to be given by a continuous Gaussian process $\cF$ taking values in the space of  holomorphic functions on $\{ |z|>1 \}$, of which we provide an explicit construction. 
The boundary values $\cW$ of $\cF$ are shown to perform an Ornstein-Uhlenbeck process on the space of distributions on the unit circle $\bbT$, which can be described as the solution to the stochastic fractional heat equation
	\[ \frac{\de}{\de t} \cW (t,\th ) = - (-\Delta )^{1/2} \cW (t,\th )  + 
	\sqrt{2}\, \xi (t, \th ) \, , \]
where $\Delta$ denotes the Laplace operator acting on the spatial component, and $\xi (t,\th )$ is a space-time white noise. 
As a consequence we find that, 
when the cluster is left to grow indefinitely, the boundary process $\cW$ converges to a log-correlated Fractional Gaussian Field, which can be realised as $(-\Delta )^{-1/4}W$, for $W$ complex White Noise on $\bbT$. 
\end{abstract}


\section{Introduction}
In 1998 the physicists M. Hastings and L. Levitov introduced a one-parameter family of continuum models for growing clusters $(K_n)_{n\geq 0}$ on the plane \cite{hastings1998laplacian}, which can be considered as an off-lattice version of discrete planar aggregation models such as the Eden model or Diffusion Limited Aggregation (DLA). 
In this paper we focus on the simplest of these models, so called HL($0$), which has proven to be already very rich from a mathematical point of view, and has received much attention in recent years
\cite{rohde2005some,norris2011weak,norris2012hastings,viklund2012scaling}.

Let $\bbD$ denote the open unit disc in the complex plane, and set $K_0 = \overline{\bbD}$. At each step, a new particle $P_n$ attaches to the cluster $ K_{n-1}$ according to the following growth mechanism.
Fix $P \subset \bbC \setminus \bbD$ to be a (non-empty) connected compact set having $1$ as a limit point, and such that the complement of $K=\overline{\bbD} \cup P$ in $\bbC  \cup \{\infty\}$ is simply connected. 
Let $D_0 = (\bbC \cup \{ \infty \}) \setminus K_0$ and  $D = (\bbC \cup \{ \infty \}) \setminus K$. Then there exist a  unique conformal isomorphism $F : D_0 \to D$ and a unique constant $c \in \bbR_+$ such that  $F(z) = e^c z + \mathcal{O}(1)$ as $|z|\to\infty$.
We think of $F$ as attaching the particle $P$ to the closed unit disc $\overline{\bbD}$ at $1$. 
The constant $c$ is called \emph{logarithmic capacity} of $K$, and can be interpreted as the expected value of $ \log | B_T|$, for $B$ planar Brownian Motion started at $\infty$, stopped at the first hitting time $T$ of $K$ (cf. Proposition \ref{teoNT}).
 
Set $G=F^{-1}$, and observe that  $G(z) = e^{-c} z + \mathcal{O}(1)$  as $|z| \to\infty$. 
Let $(\Theta_n)_{n\geq 1} $ be a sequence of i.i.d. random variables  with $\Theta_n \sim $Uniform$[-\pi ,\pi)$, and set
	\[ F_n (z) := e^{i\Theta_n} F(e^{-i\Theta_n}  z ) \, , 
	\qquad 
	G_n(z) = F_n^{-1} (z) 
	\, . \]
Then the map $F_n$ attaches the particle $P$ to the unit disc at the random point $e^{i\Theta_n}$. Define $\Phi_n(z) := F_1 \circ \cdots \circ F_n (z)$, $D_n = \Phi_n (D_0)$ and $K_n = (\bbC \cup \{ \infty \} ) \setminus D_n$. 
 We say that the conformal map $\Phi_n $ grows an HL$(0)$ cluster up to the $n$-th particle, while $\Gamma_n = \Phi_n^{-1}$ maps it out. 
Note that, by conformal invariance, choosing the attachment angles to be uniformly distributed corresponds to choosing the attachment point of the $n$-th particle according to the harmonic measure of the boundary of the cluster $K_{n-1}$ seen from infinity.

With this notation, then, one has
	\[ D_{n+1} = \Phi_n \circ F_{n+1} \circ \Gamma_n (D_n ) \, . 
	\]
This suggest the following interpretation for the attachment mechanism: given the cluster $K_n$, first map it out via $\Gamma_n$, then attach a new copy of the particle $P$ to the unit circle at a uniformly chosen point $e^{i\Theta_n}$, and finally grow back the cluster $K_n$. It is then clear that, although we are attaching identical copies of the particle $P$ at each step, the particle shape gets distorted each time by the application of the conformal map $\Phi_n$, as shown in the figure 
 below. 
\begin{figure}[!ht]
    \begin{center}
     \centering
  \mbox{\hbox{
  \includegraphics[width=0.4\textwidth]{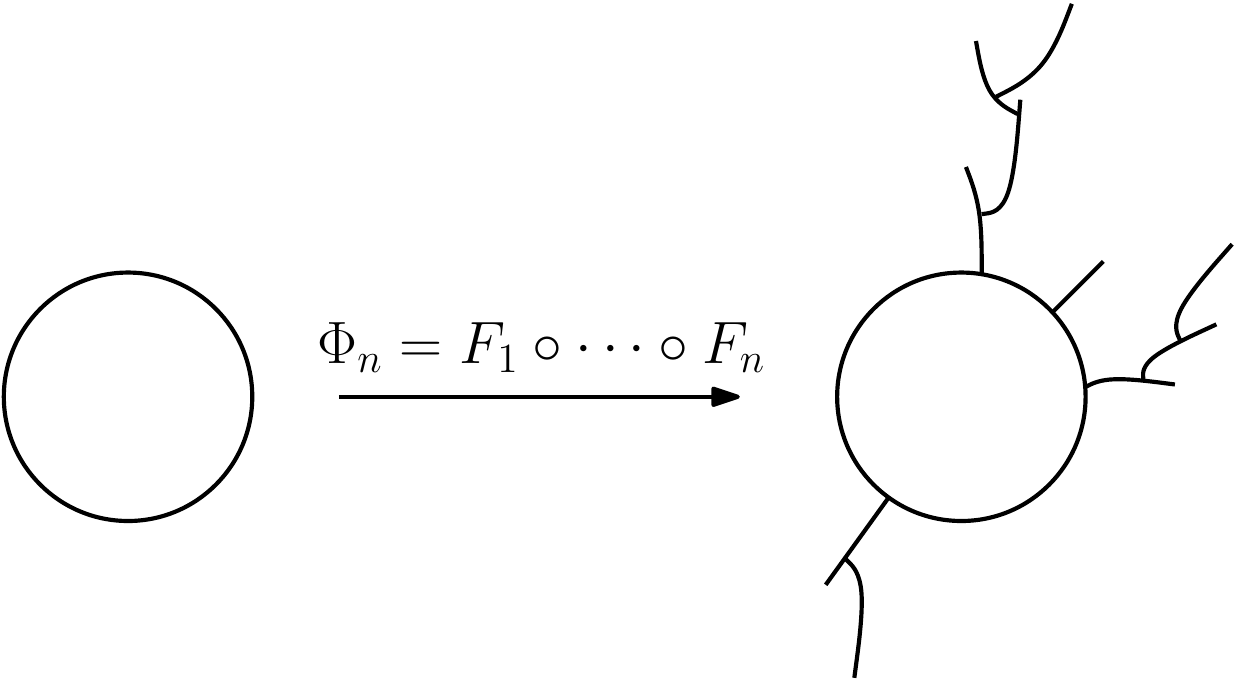}}}
       \end{center}
 \end{figure}

In \cite{norris2012hastings} J. Norris and A. Turner showed that, under mild assumption on the particle $P$, in the limit as $n\to\infty$ and $nc\to t \in \bbR_+$ the shape of HL($0$) clusters is given by a disc of radius $e^t$ centred at the origin. Moreover, with high probability each point outside the unit disc is moved radially by the cluster growth. 
This result can be rephrased as follows: given any $z \in D_0$, $\Phi_n (z) \approx e^t z$ as $n \to \infty$. 
\begin{figure}[!ht]
    \begin{center}
     \centering
  \mbox{\hbox{
  \includegraphics[width=0.3\textwidth]{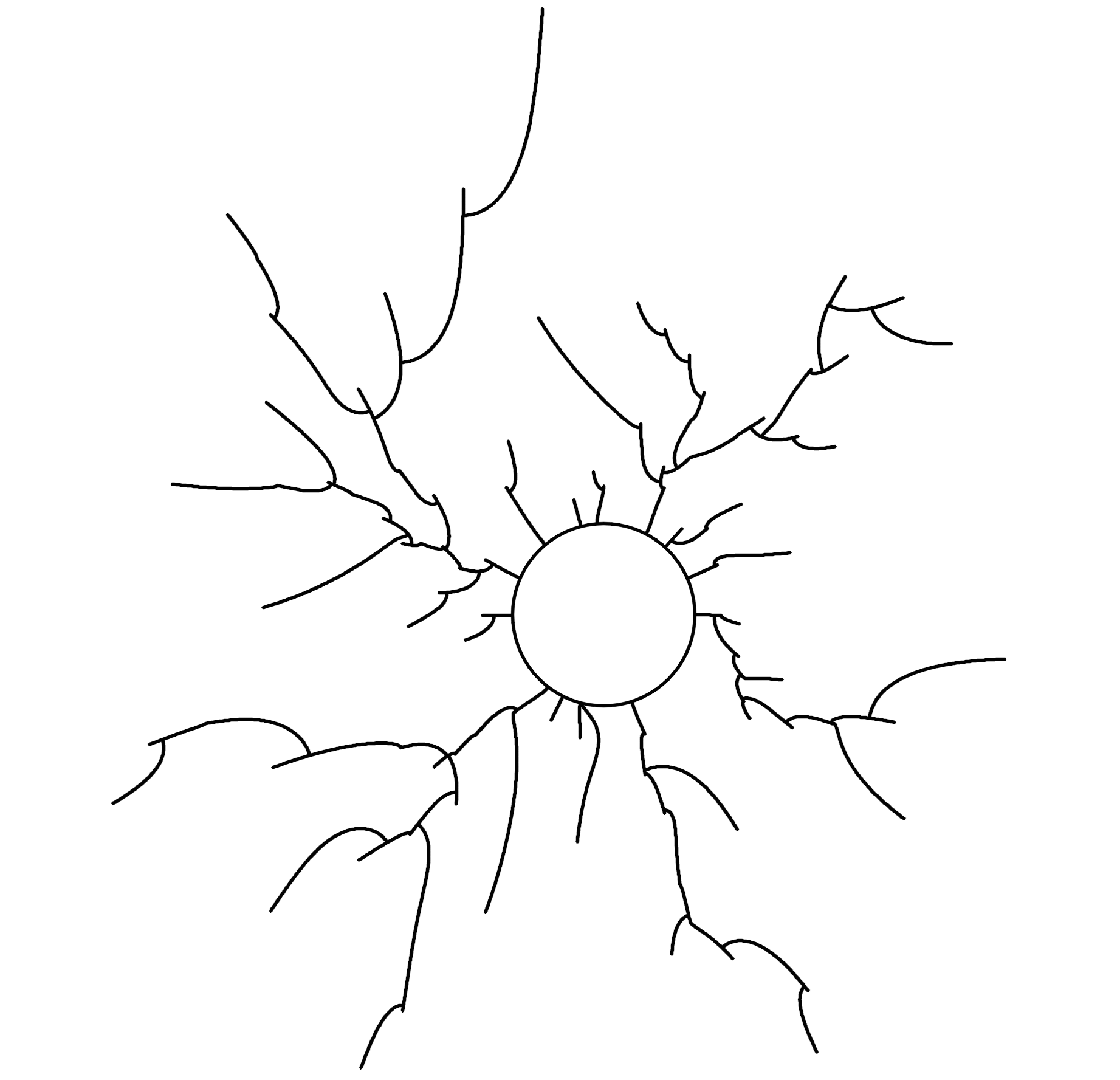} \qquad 
  \includegraphics[width=0.3\textwidth]{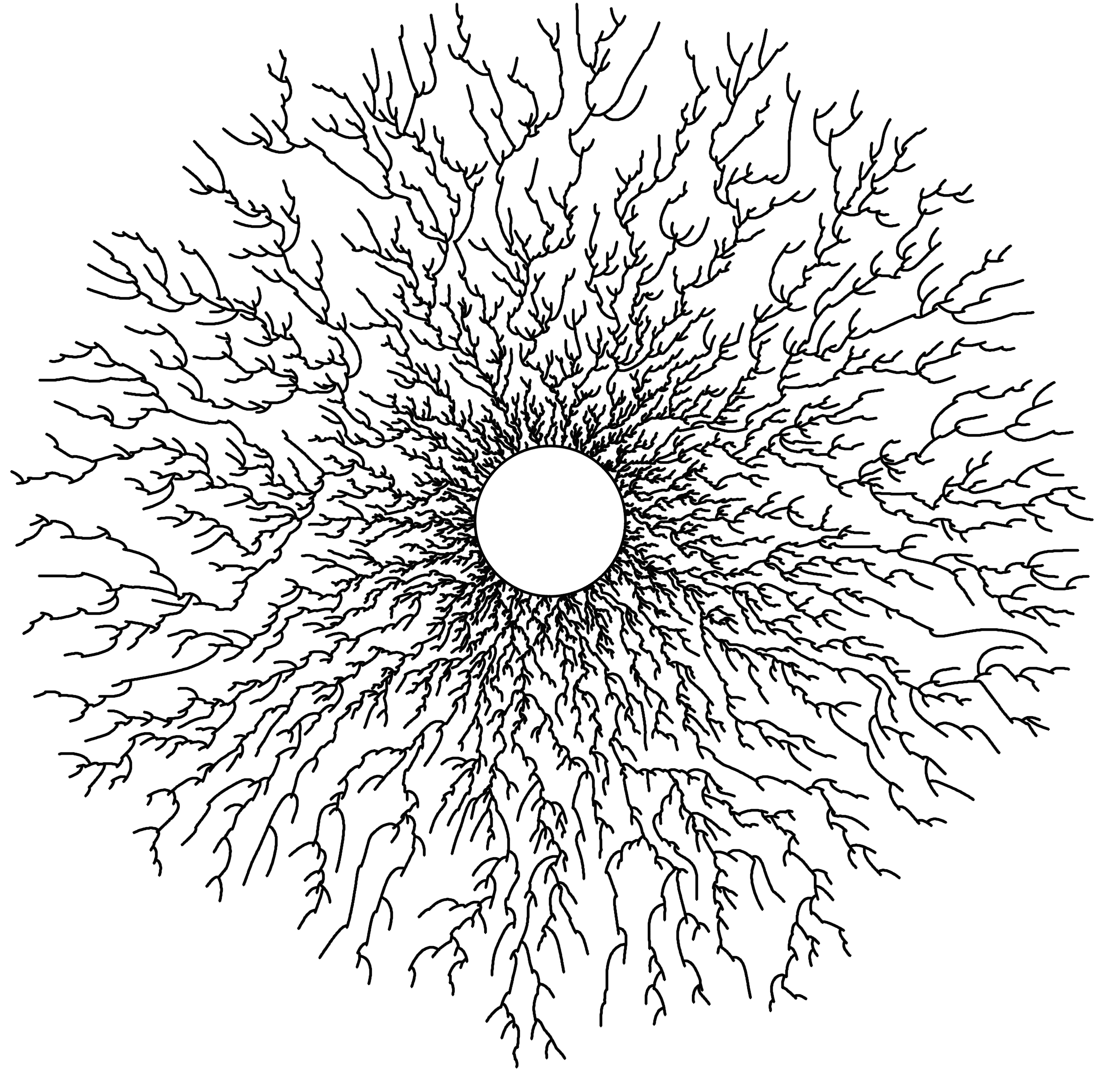} \qquad 
  \includegraphics[width=0.3\textwidth]{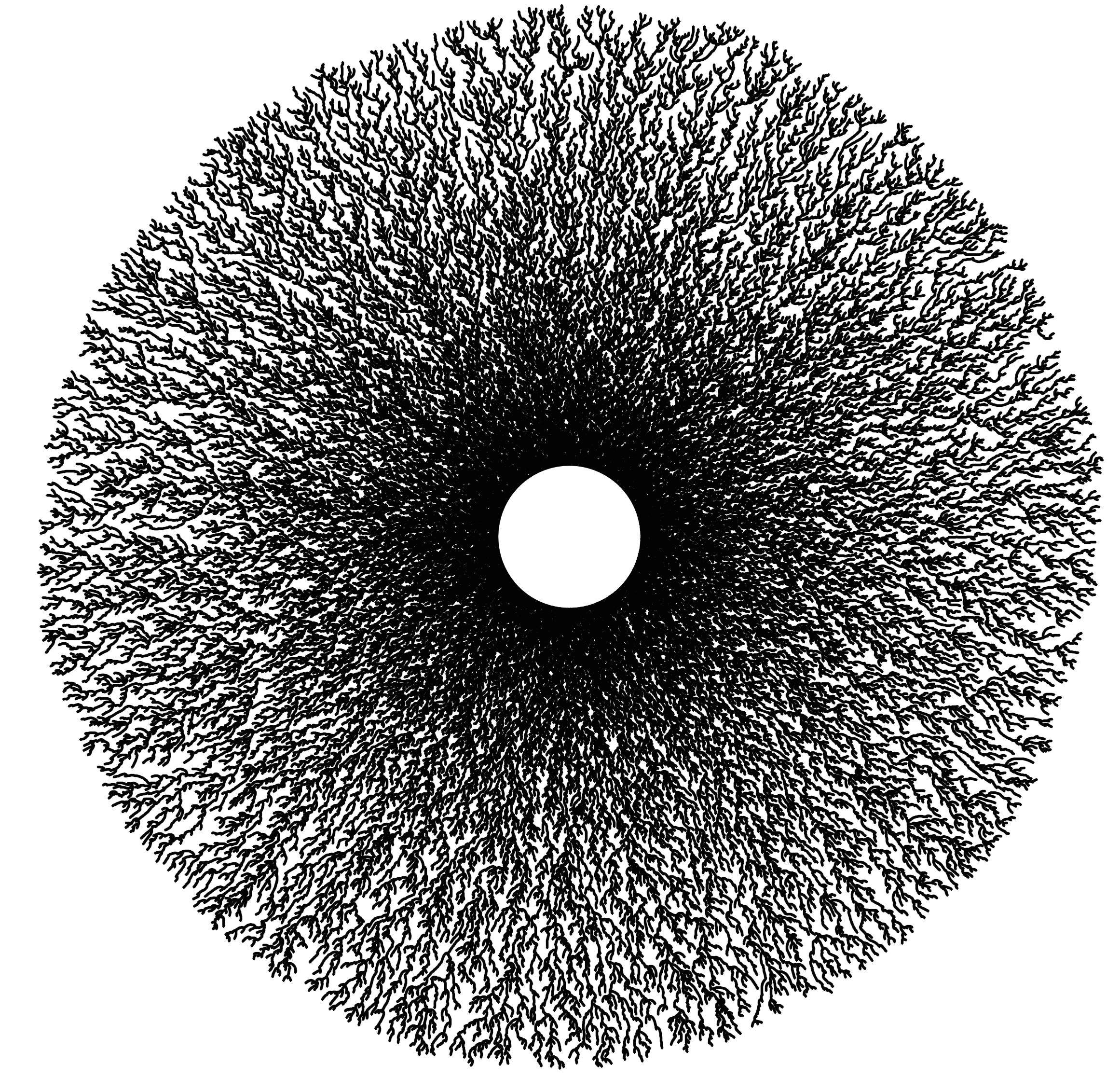}}}
       \end{center}
       \caption{Simulation of a HL($0$) cluster with $100$, $5000$ and $100000$ particles.} \label{henry}
 \end{figure}\\
In this note we investigate fluctuations of the map $\Phi_n$ around this deterministic limit. 
Our results can be divided into local and global fluctuations.

\subsection*{Local fluctuations}
Suppose we fix $z \in \bbC \setminus \overline{\bbD}$, say $z=e^{\s + ia}$ for some $\s >0 $, $a\in [-\pi , \pi )$, and look at the limiting fluctuations of $\log \Phi_n (z)$ around its mean as the HL cluster grows ($n \to \infty$) and the point $z$ approaches the unit disc radially ($\s \to 0$). We prove that, provided $\s \to 0$ slowly enough, these limiting fluctuations are Gaussian. Moreover, approaching the unit disc radially from different angles results in asymptotically independent fluctuations. Our main result for this regime is the following.

\begin{Theorem*}[Local fluctuations]\label{Tlocal}
Pick any $t>0$, and let $z = e^{ia+\s}$ for some $a \in [-\pi , \pi )$, $\s >0$. 
Then as $n \to \infty$, $nc \to t$ and $\s \to 0$ slowly enough,  
\[ \frac{\log \frac{\Phi_n (z)}{z} -nc }{\sqrt{c \log (\frac{1}{2\s})}} \longrightarrow \, \cN (0,1) \]
in distribution, where $\cN (0,1)$ denotes the law of a standard complex Gaussian random variable. Moreover, the correlation between fluctuations at two different points, say $ z=e^{ia+\s}$ and $ w=e^{ib+\s}$, vanishes in the limit, unless the angle $a-b$ converges to $0$ fast enough with $\s$. 
\end{Theorem*}
A precise statement of the above result is provided in Theorem \ref{TlocBig} of Appendix \ref{Aloc}. 
Note that the only Gaussian random field on $\bbT = \de \bbD$ having the correlation structure defined above is the one given by 
an uncountable collection of  i.i.d. $\cN(0,1)$ random variables indexed by points on the unit circle. 
Although almost surely finite at every point, this random field is very wild, in the sense that, apart from not being continuous a.s., it is not even separable (i.e. it cannot be recovered by only looking at a countable collection of points). 

\subsection*{Global fluctuations}
At the price of keeping $z$ away from the unit disc while the cluster grows, we see that the fluctuations of $\log \Phi_n$ become rather well behaved. 
In order to emphasize the dependence on $t$ in this regime, and to ultimately view the limit as a stochastic process, we now assume $nc\to 1$ as $n\to\infty$, and study the asymptotic behaviour of the map $\Phi_{\nt}$ for $t\in [0,\infty )$. 

Then, by the same techniques that allow us to prove the local fluctuations result, we can show (cf. Theorem \ref{pointwise}) that, for any \emph{fixed} $z \in \bbC \setminus \overline{\bbD}$, these fluctuations are again centred Gaussian, with  variance now depending on $|z|$ and $t$. Moreover, the correlation structure is sufficiently well behaved to enable us to prove a functional central limit theorem for $\log \Phi_\nt$ when restricted to any circle of the form $r\bbT := \{ z : |z| = r\}$, $r>1$ (cf. Theorem \ref{Tfields}). Finally, we push our analysis forward to obtain limiting fluctuations of $\log \Phi_\nt$ viewed as a c{\`a}dl{\`a}g stochastic process taking values in the space of holomorphic functions on $\bbC \setminus \overline{\bbD}$. Our main result is the following.
\begin{Theorem*}[Global fluctuations]\label{Thol_intro}
Let $\mathscr{H}$ denote the space of holomorphic functions on $\{ |z| >1\}$, and  for $n \geq 1$ set $\cF_n (t,z) = \frac{1}{\sqrt{c} } \big( \log \frac{\Phi_\nt (z)}{z} -\nt c \big)$. 
Then there exists a continuous stochastic process $\cF =(\cF(t,\cdot ))_{t\geq 0}$ taking values in $\mathscr{H}$ such that $\cF_n \to \cF$ in distribution as $n\to\infty$, with respect to the Skorokhod topology on the space of c{\`a}dl{\`a}g functions from $[0,\infty )$ to $\mathscr{H}$. 
Moreover, $\cF$ can be obtained as the holomorphic extension of its boundary values on $\{ |z|=1\}$ to the outer unit disc $\{ |z|>1\}$. These boundary values are given by a distribution-valued stochastic process $\cW=(\cW (t,\cdot ))_{t\geq 0}$, formally defined in Fourier space by
	\[ \cW(t,\th ) = \sum_{k\in \bbZ \setminus \{ 0 \} } \Big( \frac{A_k (t) + i B_k(t) }{\sqrt{2}}\Big)  \frac{e^{ik\th}}{\sqrt{2\pi}}\, ,\]
for $(A_k)_{k} $, $(B_k)_{k}$ independent collections of i.i.d. Ornstein-Uhlenbeck processes on $\bbR$, solution to
	\[ \begin{cases}
	\dd A_k (t) = -|k| A_k(t) \dd t + \sqrt{2} \,\dd \beta_k(t) \, , \\
	A_k(0)=0 
	\end{cases}
	\qquad 
	\begin{cases}
	\dd B_k (t) = -|k| B_k(t) \dd t + \sqrt{2} \,\dd \beta_k'(t)  \\
	B_k(0)=0 \, , 
	\end{cases}\]
where $(\b_k)_k$, $(\b '_k)_k$ are independent collections of i.i.d. Brownian Motions on $\bbR$.
\end{Theorem*} 
A precise statement of this result is given in Theorem \ref{Thol_new}. 
This provides an explicit construction of the limiting Gaussian holomorphic field $\cF$ and, perhaps more interestingly, of its boundary values. 
Note that, since $A_k(t) , B_k(t) \to \cN \big(0, \frac{1}{|k|}\big)$ in law as $t\to\infty$, we have 
   \[ \cW(t,\th ) \stackrel{t\to\infty}{\longrightarrow} \cW_\infty (\th ) \stackrel{(d)}{=} \sum_{k\neq 0} \frac{1}{\sqrt{|k|}} \Big( \frac{a_k + i b_k }{\sqrt{2}}\Big)  \frac{e^{ik\th}}{\sqrt{2\pi}}\, ,\]
for $(a_k)_k$, $(b_k)_k$ independent collections of i.i.d. standard Gaussian random variables  on $\bbR$.
This is (the complex version of) a well-known Fractional Gaussian Field (FGF) on the unit circle $\bbT$, and it can be realised as $\cW_\infty = (-\Delta)^{-1/4}W$, where $\Delta$ denotes the Laplace operator, and $W$ is white noise on $\bbT$. 
We remark that, in general, the FGF on $\bbT$ given by $(-\Delta )^s W$ defines a true function only for $s <-1/4$, and otherwise it takes values in the space of distributions on $\bbT$. 
It can be shown in greater generality (see the discussion in \cite{duplantier2014log} for FGFs on $\bbR^d$, $d\geq 1$) that in correspondence of the  critical value of the parameter $s$, which depends on the dimension of the underlying space, one obtains a log-correlated Gaussian field. 


\subsection*{Overview of related work} 
Fractal patterns are ubiquitous in nature, and many attempts have been made to obtain a rigorous mathematical description of their formation. 

In 1961 M. Eden \cite{eden1961two} introduced a lattice-based model of growing clusters, now called Eden model. The growth mechanism is as follows: start with a single site, and at each step choose one site on the outer boundary of the current cluster (i.e. the set of sites outside of the cluster, adjacent to at least one of the cluster's sites) uniformly at random, and add it to the cluster. 

A similar growth model, so called Diffusion Limited Aggregation (DLA), was introduced by T.A. Witten and L. Sander \cite{witten1981diffusion} in 1981. In this model, at each step a new site is sampled according to the harmonic measure of the boundary of the current cluster from $\infty$, and then added to the cluster. When, instead, one samples the new site according to the $\eta$-th power of the harmonic measure of the cluster boundary from $\infty$, one obtains the $\eta$-Dielectric Breakdown Model (DBM$_{\eta}$), which was introduced by L. Niemeyer, L. Pietronero and H.J. Wiesmann \cite{niemeyer1984fractal} in 1984. Thus, the family of DBM$_{\eta}$ models interpolates from the Eden model ($\eta =0$) to DLA ($\eta =1$). 

Although these models are simple to define, they have proven very difficult to study rigorously. Moreover, some features of large clusters, such as the Hausdorff dimension, have been found
\cite{ball1985anisotropy,meakin1987structure} to depend on the underlying lattice structure, which is somehow unsatisfactory. 

To overcome these problems, in 1988 M. Hastings and L. Levitov \cite{hastings1998laplacian} proposed a one parameter family of off-lattice models, so called HL($\a$) models for $\a \in [0,\infty )$. The case $\a =0$ has already been described in detail. When $\a >0$ the growth mechanism is the same, except that at each step the size of the  new particle $P_n$ is renormalised so that its logarithmic capacity is given by
	\begin{equation} \label{HLalpha}
	c_n = \frac{c}{| \Phi_{n-1}'(e^{i\Theta_n})|^{\a}} \, . 
	\end{equation}
If $\a =2$ this results in growing clusters of particles roughly of the same 
size, and in general \eqref{HLalpha} has the effect of attenuating the natural distortion of the $\a =0$ model. 
In \cite{hastings1998laplacian}   Hastings and Levitov argue by comparing local growth rates that the choice $\a =1$ should correspond to the Eden model, $\a=2$ to DLA and in general $\a \in (1,2)$ to DBM$_{\eta -1}$. 

Although $\a >0$ is needed for these models to be realistic, the re-normalization \eqref{HLalpha} creates long range dependences which make them very difficult to analyse. In fact, to the best of our knowledge there are no rigorous results on (the non-regularised version of) HL($\a$) models for $\a >0$. A first regularised version of HL($\a$) appears in \cite{rohde2005some}, in which S. Rohde and M. Zinsmeister obtained bounds for the Hausdorff dimension of suitably regularised clusters for $\a \in [0,2]$. More recently, a different type of regularization was considered by F. Viklund, A. Sola and A. Turner in \cite{viklund2013small}, where they showed that the limiting shape of regularised clusters is given by a disc for any $\a >0$, provided that the regularization is strong enough. 

The case $\a =0$ does not feature such long range dependences, and is much better understood. In \cite{rohde2005some} Rohde and Zinsmeister obtained a scaling limit for HL($0$) clusters as the particle size is kept fixed while $n\to\infty$. Moreover, they showed that the boundary of these limiting clusters is almost surely one-dimensional.  
More recently, in \cite{norris2011weak,norris2012hastings} Norris and Turner obtained a detailed description of HL($0$) clusters in the small particle limit. More precisely, they proved that as $n\to\infty$, $c\to 0$ and $nc \to t\in (0,\infty)$, these clusters almost surely fill a disc with radius $e^t$ centred at the origin. Moreover, they showed that the 
ancestral tree structure within the cluster converges, always in the small particle limit, to the Brownian Web \cite{fontes2004brownian}, thus providing an interesting connection between these two a priori unrelated models.

\subsection*{Organization of the paper}
We present a detailed proof of the global fluctuations result. Theorem \ref{TlocBig} on local fluctuations is obtained by the same arguments, and we leave its discussion for Appendix \ref{Aloc}. The paper is organised as follows. 
In Section \ref{S2} we collect some preliminary estimates for the basic conformal maps $F,G$. We then introduce in Section \ref{S3} our main tools, namely two sequences of backwards martingale difference arrays \eqref{XY}, and prove Theorem \ref{pointwise} on pointwise fluctuations. 
This result is then generalised in Sections \ref{S4}-\ref{Shol}, to obtain a functional central limit theorem for the process $(t,z)\mapsto \log \Phi_\nt (z)$, viewed as a stochastic process taking values in the space of 
holomorphic functions on $\{|z|>1\}$ (cf. Theorems \ref{Tfields} and \ref{Thol_new}). 
Finally, in Section \ref{explicit} we present an explicit construction of the boundary values of the  limiting fluctuation process. These are shown to be given by a distribution-valued continuous process, 
 which is rigorously defined as a Ornstein-Uhlenbeck dynamics in a suitable infinite-dimensional Hilbert space. We conclude the paper by collecting some open questions in Section \ref{Sopen}.

\subsection*{Acknowledgements} 
I am extremely grateful to James Norris for his guidance, support, and help with many of the arguments in the paper.
I am also very thankful to Alan Sola for several interesting discussions, and to Henry Jackson for providing the simulations in Figure \ref{henry}. 
\vspace{0.5cm} 

\section{Preliminary estimates} \label{S2}
Fix a particle $P$ as in the introduction, and, if $\bbD $ denotes the open unit disc, let $K_0 = \overline{\bbD}$, $K = \overline{\bbD} \cup P$. Moreover, set
 $D_0 = (\bbC \cup \{ \infty \} )\setminus K_0$,  $D = (\bbC \cup \{ \infty \}) \setminus K$. 
It follows from the Riemann Mapping Theorem that  there exists a unique conformal map\footnote{Throughout this paper, by conformal map we mean a conformal isomorphism.} $F:D_0 \to D$ such that $F(z) = e^c z + \mathcal{O}(1)$ as $|z| \to \infty$, for some constant $c\in\bbR$. We assume that the particle $P$ is regular enough so that $F$ extends continuously to the boundary of $D_0$.
 Set $G = F^{-1}$. 

Conformal maps are very rigid, and
simply from the definition of $F,G$ we can deduce the following properties:
\begin{itemize}
\item[(P1)] $|F(z)|>|z|$ for all $z \in D_0$, $|G(z)| < |z|$ for all $z \in D$, 
\item[(P2)] there exists a constant $C>0$ such that $|F(z)| \leq C |z| $ for all $z \in D_0$ and $|G(z)| \geq |z| /C$ for all $z \in D$.
\end{itemize}
Indeed, (P1) follows trivially from Schwarz lemma, while (P2) is a consequence of the prescribed behaviour at infinity for $F,G$.
\begin{Notation}
Throughout the paper, $C$ denotes a finite, positive constant which can change from line to line, and which is either absolute or only depends on the maps $F,G$. Whenever this constant depends on other parameters, say $\a , \b \ldots$, we make it explicit in the notation by using $C(\a , \b \ldots )$. 
\end{Notation}

In order to obtain finer distortion estimates for the maps $F,G$ it is often useful to relate the logarithmic capacity $c$ to geometric properties of the particle $P$. Assume the following:
	\begin{Assumption} \label{As1}
There exists $\d >0$ such that 
	 \begin{equation} \label{particle}
	 P \subseteq \{ z \in \bbC : |z-1| \leq \d \} , 
	 \quad 1+\d \in P ,
	 \quad P = \{ \bar{z} : z \in P \} . 
	 \end{equation}
	 \end{Assumption}
We regard $\d$ as measuring the diameter of the particle $P$.
\begin{Remark}
Assumption \ref{As1} is in force throughout the paper. 
\end{Remark}
The following result appears in \cite{norris2012hastings} (cf. Proposition 4.1 and Corollary 4.2 therein).
\begin{Proposition}\label{teoNT}
There exists an absolute constant $C>0$ such that, for all $z \in D$: $|z-1| >2\d$, it holds:
\begin{equation} \label{tNT}
\bigg| \log \Big( \frac{G(z)}{z} \Big) +c \bigg| \leq \frac{Cc}{|z-1|} \, , \qquad 
\bigg| \frac{\dd}{\dd z } \log  \Big( \frac{G(z)}{z} \Big) \bigg| \leq \frac{Cc}{|z-1|^2} \, . 
\end{equation}
Moreover, 
	\begin{equation}\label{CD}
	\frac{\d^2}{6} \leq c \leq \frac{3\d^2}{4}
	\end{equation} 
for $\d$ small enough.
\end{Proposition}
In light of \eqref{CD} we use $c$ and $\d$ interchangeably, and all statements are intended to hold for $c,\d$ small enough. Combining \eqref{tNT} and \eqref{CD} we deduce the following improved bound.
\begin{Corollary} \label{Pref}
There exists an absolute constant $C>0$ such that, for all $z \in D$: $|z-1| >2\d$, it holds:
	\begin{equation}\label{Gz}
	\bigg| \log \Big( \frac{G(z)}{z} \Big) +c \, \frac{z+1}{z-1}\bigg| 
	\leq \frac{Cc^{3/2}|z|}{|z-1|^2} . 
	\end{equation}
\end{Corollary}

\begin{proof}
 We follow the proof of \cite{norris2012hastings}, Proposition 4.1. 
Let $u,v$ denote the real and imaginary part of $\log \frac{G(z)}{z} $ respectively, so that they are harmonic functions on $D$. Then by optional stopping $u(z) = -\bbE ( \log |B_T|) <0$, $T$ being the first hitting time of $K$ for a Brownian Motion $B$ starting from $z$. 
Introduce the particle $P_1 \supset P$ defined by $P_1 = \big\{ z\in D_0 : \big| \frac{z-1}{z+1}\big| \leq r \big\}$, for $r = \d /(2-\d )$, and set $D^1 = (\bbC \cup \{\infty\})\setminus (\overline{\bbD} \cup P_1)$. Then the unique conformal map $G_1:D^1\to D_0$ satisfying $G_1(\infty ) =\infty$ and $G_1'(\infty ) >0$ is given by 
	\[ G_1(z)=\frac{z(\g z-1)}{z-\g} \qquad \mbox{ for } \; \g = \frac{1-r^2}{1+r^2}.\]
Set $F_1 = G_1^{-1}$, and $A = \{ z \in \de P_1 : |z|>1 \}$. Then $G_1(A) = \{ e^{i\th} : |\th |<\th_0 \}$ with $\th_0 = \cos^{-1}\g$. 
\begin{figure}[!ht]
    \begin{center}
     \centering
  \mbox{\hbox{
  \includegraphics[width=0.45\textwidth]{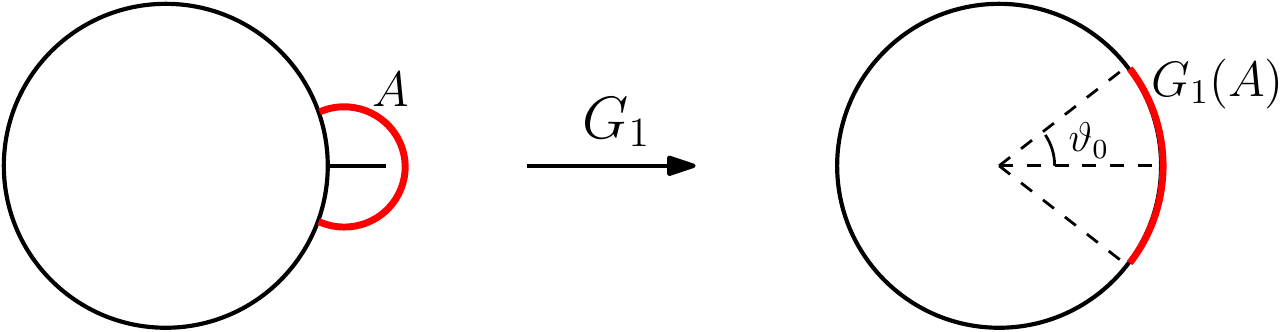}}} 
       \end{center}
 \end{figure} Moreover, $u \circ F_1$ is harmonic and bounded on $D_0$ and, using that $\frac{1}{2\pi}\int_{|\th|\leq \th_0} (u\circ F_1) (e^{i\th})  \dd \th = -c$,  the optional stopping theorem  yields 
	\[ (u\circ F_1 )(z) = -c + \frac{1}{2\pi} \int_{|\th|\leq \th_0} (u\circ F_1) (e^{i\th}) \Re \Big( \frac{2e^{i\th}}{z-e^{i\th}} \Big) \dd \th \, . \]
It follows that
	\[  (u\circ F_1 )(z) + c + \Re \Big( \frac{2c}{z-1} \Big) = 
	\frac{1}{2\pi} \int_{|\th|\leq \th_0} (u\circ F_1) (e^{i\th}) \Big[ \Re \Big( \frac{2e^{i\th}}{z-e^{i\th}} \Big) - \Re \Big( \frac{2}{z-1} \Big) \Big]\dd \th \, . \]
Using \eqref{CD}, then, we find $1-\cos \th_0 = 1-\g \asymp \d^2/2 \leq C \sqrt{c}$, from which 
	\[ \bigg| \Re \Big( \frac{2e^{i\th}}{z-e^{i\th}} \Big) - \Re \Big( \frac{2}{z-1} \Big) \bigg| 
	\leq \frac{2|z| |1-e^{i\th_0}| }{|z-e^{i\th}| | z-1|} 
	\leq \frac{C|z| \sqrt{c} }{|z-e^{i\th}| | z-1|} 
	\leq \frac{C|z| \sqrt{c} }{\textrm{dist}(z,G_1(A))^2} \, . \]
Since $(u\circ F_1)(e^{i\th})<0$ for all $\th $ in the integration range, we find
	\[ \bigg| (u\circ F_1 )(z) + c + \Re \Big( \frac{2c}{z-1} \Big) \bigg| 
	\leq \frac{C|z| \sqrt{c} }{\textrm{dist}(z,G_1(A))^2} 
	\bigg| \frac{1}{2\pi} \int_{|\th|\leq \th_0} (u\circ F_1) (e^{i\th}) \dd \th \bigg| 
	= \frac{C|z| c^{3/2} }{\textrm{dist}(z,G_1(A))^2} \, . \]
Now take $G_1(z)$ in place of $z$, to get 
	\begin{equation}\label{long}
	 \begin{split} 
	\bigg| u(z) + c + \Re \Big( \frac{2c}{z-1} \Big) \bigg| &
	\leq \bigg| u(z) + c + \Re \Big( \frac{2c}{G_1(z)-1} \Big) \bigg| 
	+ \bigg| \Re \Big( \frac{2c}{G_1(z)-1} \Big) - \Re \Big( \frac{2c}{z-1} \Big) \bigg| 
	\\ & \leq \frac{C|z| c^{3/2} }{\textrm{dist}(G_1(z),G_1(A))^2} 
	+ \frac{2c|G_1(z)-z| }{|G_1(z)-1| |z-1|} \, ,
	\end{split} 
	\end{equation}
where in the second inequality we have used that $|G(z)| < |z|$.
Using the explicit expression for $G_1$, one shows that $|G_1(z)-1| > |z-1|/2$ and $|G_1(z)-z|  \leq C\sqrt{c} |z| $ for $\d$ small enough, from which 
	\begin{equation}\label{facile}
	  \frac{2c|G_1(z)-z| }{|G_1(z)-1| |z-1|}  < \frac{C |z| c^{3/2}}{|z-1|^2}\, . 
	  \end{equation}
Moreover, reasoning as in \cite{norris2012hastings} one shows that there exists an absolute constant $C_1$ such that $\textrm{dist}(G_1(z),G_1(A)) \geq |z-1|/C_1$. Putting this together with \eqref{long} and \eqref{facile} we finally obtain
	\[ \bigg| u(z) + c + \Re \Big( \frac{2c}{z-1} \Big) \bigg| 
	\leq \frac{C|z| c^{3/2} }{|z-1|^2}  \]
for all $z\in D$ such that $|z-1|>2\d$. 

Now note that $u(z) +c + \Re \big( \frac{2c}{z-1} \big) $ and $v(z) + \Im \big( \frac{2c}{z-1} \big) $ are the real and imaginary part of the holomorphic function $z \mapsto \log \frac{G(z)}{z} + c + \frac{2c}{z-1} $ on $D \cap \{ z : |z-1|>2\d\}$. It then follows by Cauchy's integral formula that
	\[\bigg|  \nabla \bigg( v (z) + \Im \Big( \frac{2c}{z-1} \Big) \bigg) \bigg| 
	= \bigg|  \nabla \bigg( u (z) + c+ \Re \Big( \frac{2c}{z-1} \Big) \bigg) \bigg| 
	\leq \frac{C|z| c^{3/2} }{|z-1|^3} \, . \]
Finally, using that $\big| v(z) + \Im \big( \frac{2c}{z-1} \big) \big| \to 0$ as $|z|\to\infty$, we get
	\[ \bigg|   v (z) + \Im \Big( \frac{2c}{z-1} \Big) \bigg| 
	\leq \int_0^\infty
	\bigg|  \nabla \bigg( v (z + s(z-1)) + \Im \Big( \frac{2c}{z+s(z-1)-1} \Big) \bigg) \bigg| 
	\cdot |z-1| \dd s 
	\leq \frac{C|z| c^{3/2} }{|z-1|^2} \, , \]
which concludes the proof.
\end{proof}

We combine Proposition \ref{teoNT} and Corollary \ref{Pref} to obtain corresponding estimates for the function $F$, which are collected below.
\begin{Corollary} \label{Fcor}
There exists a constant $C>0$ such that, for all $z\in D_0 $ with $ |F(z)-1|>2\d$, it holds:
\[ | F(z) - e^c z | \leq \frac{C c |z|}{|z|-1} \, , 
\qquad \quad 
\bigg| \log \Big( \frac{F(z)}{z} \Big) -c \, \frac{z+1}{z-1}\bigg| 
	\leq \frac{Cc^{3/2}|z|^2}{(|z|-1)^3} .\]
\end{Corollary} 
\begin{proof}
For the first inequality, note that \eqref{tNT} readily implies that 
\begin{equation}\label{Gconseq}	
| G(z) - e^{-c} z | \leq \frac{Cc|z|}{|z-1|} 
\end{equation}
for all $z \in D : |z-1| > \d $, and $\d$ small enough. Therefore, by setting $w = F(z)$ and using (P1)-(P2), we obtain 
	\begin{equation}\label{barc} 
	| F(z) - e^c z | = e^c | G(w) - e^{-c} w| 
	\leq \frac{Cc|w|}{|w-1|} \leq \frac{C' c |z|}{|z|-1} \, , 
	\end{equation}
for all $z \in D_0$ such that $|F(z)-1|>2\d$ and $\d$ small enough, as claimed. 

For the second inequality, note that, since $|F(z) -1| >  2\d$ by assumption, we can use \eqref{Gz} to get 
	\[ \begin{split}
	\bigg| \log  \frac{F(z)}{z} -c \, \frac{z+1}{z-1}\bigg|  & 
	 \leq \bigg| \log \frac{G(w)}{w} +c \, \frac{w+1}{w-1}\bigg|
	+ 2c \bigg| \frac{1}{w-1} - \frac{1}{G(w)-1} \bigg|    \\
	& \leq \frac{Cc^{3/2} |w|}{|w-1|^2} + \frac{2c|G(w)-w|}{|w-1| |G(w)-1| } 
	 \leq \frac{Cc^{3/2} |z|}{(|z|-1)^2} + \frac{2c|G(w)-w|}{(|z|-1)^2 } \, . 
	\end{split} \]
Moreover, it follows from  \eqref{barc} that
	\[ |G(w)-w| \leq  
	|F(z) -e^{c} z | + (1-e^{-c})|w| \leq 
	\frac{Cc|z|}{|z|-1 } + Cc |z| \leq \frac{Cc|z|^2}{|z|-1 } \]
for $c$ small enough. Putting all together, we end up with 
	\[  \bigg| \log  \frac{F(z)}{z} -c \, \frac{z+1}{z-1}\bigg|   
	\leq \frac{Cc^{3/2} |z|}{(|z|-1)^2}  \bigg( 1 + \frac{\sqrt{c} |z|}{|z|-1} \bigg) 
	\leq \frac{2Cc^{3/2} |z|^2}{(|z|-1)^3} \, \]
for $c$ small enough, as claimed. 
\end{proof}

\vspace{3mm}

\section{Pointwise fluctuations} \label{S3}
In this section we prove that, fixed $t\geq 0$ and  $z \in \bbC \setminus \overline{\bbD}$,  the fluctuations of $\log \Phi_\nt (z) $ around its mean are given by a complex Gaussian random variable, whose variance is independent of $\textrm{Arg}(z)$. 
\begin{Notation}
Throughout the paper $\cN (\mu , \s^2 )$ denotes the Gaussian distribution on $\bbR$ with mean $\mu$ and variance $\s^2$.
\end{Notation}
Our main result is the following.

\begin{Theorem}\label{pointwise}
Fix any $t>0$. Pick $a \in [-\pi , \pi )$, $\s >0$, and let $z = e^{ia+\s}$. 
Define $v^2_t(\s) :=  \log \frac{1-e^{-2(\s +t)} }{1-e^{-2\s }} $, and let $\cF^\s (t, e^{ia}) $ be a complex Gaussian random variable with i.i.d.  real and imaginary part,  distributed according to $\cN ( 0 ,  v^2_t(\s )) $. 
Then it holds:
	\[ \frac{1}{\sqrt{c}} \bigg( \log \frac{\Phi_\nt (z)}{z} -\nt c \bigg) \longrightarrow \, \cF ^\s (t,e^{ia}) \]
in distribution as $n \to \infty$, $c \to 0$ and $nc \to 1$. 
\end{Theorem}


We prove Theorem \ref{pointwise} for $\s \leq 1$, which we assume without further notice. This entails no loss of generality (see discussion in Section \ref{Shol}), and it has the advantage of slightly simplifying the notation. 

\begin{Remark}
From this point onwards, with the exception of Appendix \ref{Aloc}, whenever we write $n\to\infty$ or $c\to 0$ we mean that $n\to\infty , c\to 0$ and $nc \to 1$. 
\end{Remark}

Our main tool for the proof of Theorem \ref{pointwise} consists of two sequences of backwards martingale difference arrays, that we now define. Note that $\big| \frac{\Phi_\nt (z)}{z} \big|>1$ for all $z \in D_0$, so $\log \frac{\Phi_\nt (z)}{z}$ defines a holomorphic function on $D_0$. We fix the branch of the logarithm by requiring that $\log \frac{\Phi_\nt (z)}{z} \to c$ as $z \to \infty$. 
For $\s >0$ and $a \in [-\pi , \pi )$ as in Theorem \ref{pointwise}, then, we find:
\[ \log \frac{\Phi_\nt (e^{ia+\s})}{e^{ia+\s}} = 
\sum_{k=1}^\nt \log \frac{ F_k \circ F_{k+1} \circ \cdots \circ F_\nt (e^{ia+\s})}{ F_{k+1} \circ \cdots \circ F_\nt (e^{ia+\s})} 
=
\sum_{k=1}^\nt \log \frac{ F(e^{-i\Theta_k} Z_{k,\nt }^\s (a) )}{e^{-i\Theta_k} Z_{k,\nt}^\s (a) } \, , \]
where $Z_{k,\nt}^\s (a)  := F_{k+1} \circ \cdots \circ F_\nt (e^{ia+\s})$. Moreover, if $\mathscr{F}_{k,n} := \sigma ( \Theta_k , \Theta_{k+1} \ldots \Theta_n )$, then 
	\[ \bbE \bigg( \log \frac{ F(e^{-i\Theta_k} Z_{k,\nt}^\s (a) )}{e^{-i\Theta_k} Z_{k,\nt}^\s (a) } 
	\bigg| \mathscr{F}_{k+1,\nt} \bigg) = \frac{1}{2\pi} \int_{-\pi}^\pi 
	\log \frac{ F(e^{-i\th} Z_{k,\nt}^\s (a) )}{e^{-i\th} Z_{k,\nt}^\s (a) } \dd \th 
	= \lim_{|z| \to \infty } \log \frac{F(z)}{z} = c \, .\]
We therefore set 
\begin{equation}\label{XY}
X_{k,\nt}^\s (a) \! = \frac{1}{\sqrt{c}}  \bigg( \! \log \bigg|\frac{ F(e^{-i\Theta_k} Z_{k,\nt}^\s (a) )}{e^{-i\Theta_k} Z_{k,\nt}^\s (a) }\bigg|  -c \bigg) , 
\quad \; \,
Y_{k,\nt}^\s (a) \! = \frac{1}{\sqrt{c}} \textrm{Arg}  \bigg( \frac{ F(e^{-i\Theta_k} Z_{k,\nt}^\s (a) )}{e^{-i\Theta_k} Z_{k,\nt}^\s (a) } \bigg) 
\end{equation}
where $\textrm{Arg}(z) \in [-\pi , \pi )$. The above computation then shows that  $(X_{k,\nt}^\s (a) )_{ k \leq \nt}$ and $ (Y_{k,\nt}^\s (a) )_{ k \leq \nt}$ form sequences of backwards martingale arrays with respect to the same filtration $(\mathscr{F}_{k,n})_{k \leq \nt}$ (see \cite{billingsley1995probability}, Section 35 for the definition of martingale arrays).
Moreover, 
	\[ \sum_{k=1}^\nt \big(  X_{k,\nt}^\s (a)  +i  \, Y_{k,\nt}^\s (a) \big) = \frac{1}{\sqrt{c}} \bigg( \log \frac{\Phi_\nt (e^{ia+\s})}{e^{ia+\s}} -\nt c \bigg) \, . \]
\begin{Remark}
Throughout the paper we identify $\bbC$ with $\bbR^2$, and often refer to complex random variables as random vectors and vice versa, depending on which point of view we seek to emphasize.
\end{Remark}

The following result appears in  \cite{mcleish1974dependent}, Corollary 2.8, as a Central Limit Theorem for (forward) martingale difference arrays (see also \cite{billingsley1995probability}, Theorem 35.12). It is straightforward to adapt the proof to backwards martingale difference arrays, so that we have the following. 

\begin{Theorem} \label{ML_CLT}
Let $(\mathcal{X}_{k,n})_{1\leq k \leq n}$  be a backwards martingale difference array with respect to $\mathscr{F}_{k,n} = \s ( \mathcal{X}_{k,n} \ldots \mathcal{X}_{n,n} )$.
Let $S_{k,n} = \sum_{j=k }^n \mathcal{X}_{j,n}$. Assume that:
\begin{itemize}
\item[(I)] for all $\eta >0$,  $\ds \sum_{k=1}^n \mathcal{X}_{k,n}^2 \mathds{1}(|\mathcal{X}_{k,n}|>\eta ) \to 0$ in probability as $n\to\infty$,
\item[(II)] $\ds \sum_{k=1}^n \mathcal{X}_{k,n}^2 \to s^2$ in probability as $n\to\infty$, for some $s^2 >0$.
\end{itemize}
Then $S_{n,n} $ converges in distribution to $\cN (0,s^2 )$.  
\end{Theorem}

Note that the above theorem is concerned with scalar random variables, while we have 
$2$--dimensional vectors. In order to reduce to the scalar case, recall that by the Cram{\'e}r--Wold Theorem (cf. \cite{durrett2010probability}, Theorem 3.9.5) it suffices to prove convergence in distribution of all linear combinations of the vector entries. To this end, pick any $\a , \beta \in \bbR$, and note that by linearity $\big( \a X_{k,\nt}^\s (a)  + \b Y_{k,\nt}^\s (a) \big)_{ k \leq \nt}$ is again a backwards martingale difference array with respect to the filtration $(\mathscr{F}_{k,\nt})_{ k \leq \nt}$. 
We are going to apply Theorem \ref{ML_CLT}  to  this linear combination.  To this end, we collect here some estimates for $(X^\s_{k,\nt}(a))$ and $(Y^\s_{k,\nt}(a))$. Since $a$ and $\s$ are fixed, we omit them from the notation throughout this section.

\begin{Lemma} \label{badE}
There exists a constant $C>0$ such that, for $c$ small enough, it holds $|X_{k,\nt}| \leq C/\sqrt{c}$, $|Y_{k,\nt}| \leq C/\sqrt{c}$ for all $n\geq 1$, $k\leq \nt$.
\end{Lemma}
\begin{proof}
It follows from (P1) that $\log \big| \frac{F(e^{-i\Theta_k} Z_{k,\nt})}{e^{-i\Theta_k} Z_{k,\nt}} \big| >0$, from which $|X_{k,\nt}| >-\sqrt{c}$. Moreover, (P2) gives $ |X_{k,\nt}| \leq C/\sqrt{c} + \sqrt{c} \leq 2C/\sqrt{c}$ for some constant $C>0$ and $c$ small enough. Finally, since $\textrm{Arg} \big( \frac{F(e^{-i\Theta_k} Z_{k,\nt})}{e^{-i\Theta_k} Z_{k,\nt}} \big) \in [-\pi , \pi)$, we have $|Y_{k,\nt}| \leq \pi / \sqrt{c} $. 
\end{proof}

Much better estimates can be obtained by restricting to a certain \emph{good event}, which is shown in  \cite{norris2012hastings} to have high probability for large $n$. The following result identifies this event. 

\begin{Theorem}[\cite{norris2012hastings}, Proposition 5.1] \label{teoJ}
Fix a positive integer $m$ and a constant $\e >0$. For $n\leq m$ define the events 
	\[ \begin{split}
	E_n (\e )   : & = \{ | e^{-cn} \Phi_n (z) - z | < \e e^{6\e} \mbox{ for all } z : |z| \geq e^{5\e} \}
	\\ & \cap 
	\{ | e^{cn} \Gamma_n (z) - z | < \e e^{5\e +cn} \mbox{ for all } z : |z| \geq e^{cn+4\e} \} \, ,
	\end{split} \]
and set $E (m,\e ):= \bigcap_{n=1}^m E_n (\e )$. Then it holds
	\[ \bbP (E (m,\e )^c ) \leq C (m+\e^{-2}) e^{-\e^3 / (Cc)} \]
for some constant $C>0$. In particular, by setting $m = \lfloor \d^{-6} \rfloor$ and $\e = \d^{2/3} \log (1/\d )$, one obtains that $ \d^{-k} \bbP (E (m,\e )^c ) \to 0  $
as $\d \to 0$, for any $k\geq 0$. 
\end{Theorem}

We refer to $E(m,\e)$ as the \emph{good event}. 
\begin{Remark}
Without further notice, we take $m = \lfloor \d^{-6} \rfloor$ and $\e = \d^{2/3} \log (1/\d )$ as in the last part of the above theorem, so that $\e \gg \d$ and $n \asymp \d^{-2} \ll m$.
\end{Remark}

\begin{Lemma} \label{LEgoodE}
Assume that $\s \gg \d$. 
Then there exists a constant $C(t)>0$ such that, for $n$ large enough, on the good event $E(m,\e )$ it holds 
	\begin{equation} \label{goodE}
	\max_{k\leq \nt} \Big\{ |X_{k,\nt}| \vee |Y_{k,\nt}| \Big\} \leq C(t) \frac{\sqrt{c}}{\s} \, . 
	\end{equation}
\end{Lemma} 
\begin{proof}
For any $k\leq \nt$ we have
	\[\begin{split} 
	 \max \big\{ |X_{k,\nt}| ; |Y_{k,\nt}| \big\} & \leq 
	\frac{1}{\sqrt{c} } \bigg(  \big| \log F(e^{-i\Theta_k} Z_{k,\nt}  ) - 
	\log (e^{-i\Theta_k} Z_{k,\nt})  \big| +c \bigg) \\ & 
	\leq \frac{1}{\sqrt{c}} \bigg[ \Big( \sup_{|\xi|\geq e^\s} \frac{1}{|\xi |}\Big) 
	\cdot \big| F(e^{-i\Theta_k} Z_{k,\nt} ) - e^{-i\Theta_k} Z_{k,\nt}\big| + c \bigg] \, , 
	\end{split}\]
where the last inequality follows from the mean values theorem, and the fact that  $| F(e^{-i\Theta_k} Z_{k,\nt}) | > |e^{-i\Theta_k} Z_{k,\nt}| > e^\s$ almost surely by (P1). 
Now note that $|F(e^{-i\Theta_k} Z_{k,\nt}) -1| > |Z_{k,\nt}|-1 \geq \s \gg 2\d $, so by Corollary \ref{Fcor} we have 
	\begin{equation} \label{Fz}
	 | F(z) - z | \leq |F(z) - e^c z | + (e^c -1) |z| \leq \frac{Cc|z|}{|z|-1} + 2c |z| \leq 
	\frac{2Cc|z|^2}{|z|-1} \, 
	\end{equation}
for $z=e^{-i\Theta_k} Z_{k,\nt}$.
Moreover, since we are on $E(m,\e )$, there exists a constant $C(t)$ depending only on $t$ such that $e^\s <|Z_{k,\nt}| \leq C(t)$. This, together with \eqref{Fz}, yields 
	\[ \begin{split} 
	 \max_{k\leq \nt} \Big\{ |X_{k,\nt}| \vee |Y_{k,\nt}| \Big\} & \leq \frac{1}{\sqrt{c} }\bigg(  \frac{2Cc | Z_{k,\nt}|^2 }{|Z_{k,\nt}| -1} + c \bigg) \leq \frac{C(t) \sqrt{c}}{e^\s -1} + \sqrt{c} 
	 \leq 2 C(t) \cdot \frac{\sqrt{c}}{\s} \,  
	\end{split}\]
as claimed.
\end{proof}

We now  make use of the above bounds to prove that the backwards martingale difference array  $(\cX_{k,\nt} )_{k\leq \nt } $, with $\cX_{k,\nt} :=  \a X_{k,\nt}+ \b Y_{k,\nt}$, satisfies Assumptions (I)-(II) of Theorem \ref{ML_CLT}. In doing so, we provide an explicit formula for the limiting variance. 

\begin{Lemma} \label{le3}
Assume that $\s \gg \d$. Then
for all $\eta >0$ it holds
$\ds  \sum_{k=1}^\nt  \cX_{k,\nt}^2 \mathds{1} ( |\cX_{k,\nt}| >\eta  ) \to 0 $
in probability as $n\to \infty$. 
\end{Lemma}
\begin{proof}
For any $\epsilon >0$ we have:
	\[ \begin{split} 
	\bbP \bigg(   \sum_{k=1}^\nt  \cX_{k,\nt}^2  \mathds{1} (
	 |\cX_{k,\nt}| & >\eta )  >\epsilon  \bigg)  
	\leq \bbP \bigg( \max_{1\leq k \leq \nt} 
	  |\cX_{k,\nt}| >\eta \bigg) 
	  \leq \frac{1}{\eta} \bbE \bigg( \max_{1\leq k \leq \nt} 
	  |\cX_{k,\nt}|  \bigg) 
	  \\ & 
	  = \frac{1}{\eta} \bbE \bigg( \max_{1\leq k \leq \nt} 
	  |\cX_{k,\nt}| \, ; E(m,\e )^c  \bigg)  + 
	  \frac{1}{\eta} \bbE \bigg( \max_{1\leq k \leq \nt} 
	  |\cX_{k,\nt}| \, ; E(m,\e ) \bigg) \, . 
	  \end{split}  
	\]
The fact that  the first term in the r.h.s. converges to zero as $n\to\infty$ follows from Lemma \ref{badE} and Theorem \ref{teoJ}, while  
convergence to zero of the second term is a straightforward consequence of Lemma \ref{LEgoodE}. 
\end{proof}

We now concentrate on Assumption (II). The first step consists in replacing condition (II) with a more convenient one, involving conditional second moments. 
The following result shows that, provided $\s$ is large enough with respect to $c$, this is allowed.
 
\begin{Lemma} \label{replace}
Assume $\s \gg \sqrt{\d} $, and that
 $\ds \sum_{k=1}^\nt \bbE ( \cX_{k,\nt}^2 | \mathscr{F}_{k+1,\nt} ) \to s^2$ in probability as $n\to\infty$ for some $s^2>0$. Then also 
$\ds \sum_{k=1}^\nt  \cX_{k,\nt}^2 \to s^2$ in probability as $n\to\infty$. 
\end{Lemma}

\begin{proof}
Let $M_{k,\nt} :=  \cX_{k,\nt}^2 - \bbE (   \cX_{k,\nt}^2 | \mathscr{F}_{k+1,\nt})$. 
It is readily checked that $(M_{k,\nt})_{k\leq \nt}$ is a backwards martingale difference array with respect to the filtration $(\mathscr{F}_{k,\nt} )_{k\leq \nt}$. We aim to show that for any $\eta >0$ it holds $\ds \bbP \Big( \Big| \sum_{k=1}^\nt M_{k,\nt} \Big| >\eta \Big) \to 0$ as $n\to\infty$. 
Indeed,
	\[ \begin{split} 
	 \bbP \Big( \Big| \sum_{k=1}^\nt M_{k,\nt} \Big| >\eta \Big) 
	 & \leq \frac{1}{\eta^2} \bbE \bigg( \bigg[ \sum_{k=1}^\nt M_{k,\nt} \bigg]^2 \bigg) 
	= \frac{1}{\eta^2}  \sum_{k=1}^\nt \bbE (M_{k,\nt}^2) 
	  \stackrel{(*)}\leq  \frac{1}{\eta^2}  \sum_{k=1}^\nt  \bbE ( \cX_{k,\nt}^4 ) \\
	 & =\frac{1}{\eta^2} 
	 \sum_{k=1}^\nt \bbE ( \cX_{k,\nt}^4 ; E(m,\e )^c )
	 + \frac{1}{\eta^2} \sum_{k=1}^\nt \bbE ( \cX_{k,\nt}^4 ; E(m,\e ) )  \, . 
	 \end{split} \]
Above, 
$(*)$ follows from the general inequality $\bbE((X-\bbE(X))^2) \leq \bbE(X^2)$, which we apply to each term  with respect to $\bbE ( \, \cdot \, | \mathscr{F}_{k,\nt} )$.
The fact that both terms in the r.h.s. converge to zero as $n\to\infty$ is now a consequence of the bounds for $|X_{k,\nt } |$ and $|Y_{k,\nt}|$, and hence for $|\cX_{k,\nt}|$, obtained in Lemmas \ref{badE}--\ref{LEgoodE}. 
\end{proof}

In light of the above result, it remains to compute the limit in probability of 
	\[ \begin{split} 
	\sum_{k=1}^\nt \bbE ( \cX_{k,\nt}^2 | \mathscr{F}_{k+1,\nt} )
	= \a^2 \sum_{k=1}^\nt & \bbE ( X_{k,\nt}^2 |  \mathscr{F}_{k+1,\nt} )
	 + \b^2 \sum_{k=1}^\nt \bbE ( Y_{k,\nt}^2 | \mathscr{F}_{k+1,\nt} )
	\\ & 
	+ 2\a \b \sum_{k=1}^\nt \bbE ( X_{k,\nt} Y_{k,\nt} | \mathscr{F}_{k+1,\nt} ) \, ,
	\end{split}\]
and prove that it coincides with $(\a^2 + \b^2) v^2_t(\s )$, where $v^2_t(\s ) $ is the limiting variance introduced in Theorem \ref{pointwise}.
The following result shows that, in fact, it suffices to compute the limit of the first term in the r.h.s. above. 

\begin{Proposition} \label{simpler}
It holds $\bbE (X_{k,\nt}^2 | \mathscr{F}_{k+1,\nt} ) = \bbE (Y_{k,\nt}^2 | \mathscr{F}_{k+1,\nt} )$, and $\bbE( X_{k,\nt} Y_{k,\nt} | \mathscr{F}_{k+1,\nt} ) =0$ almost surely for all $k\leq \nt$. 
\end{Proposition}

\begin{proof} 
All equalities  in this proof are intended to hold almost surely. 
Introduce the holomorphic function $f(z) := \frac{1}{\sqrt{c}} \big( \log \frac{F(z)}{z} -c \big)$ defined for $|z|>1$, so that 
	\[ \Big[ f(e^{-i\Theta_k} Z_{k,\nt} ) \Big]^2 
	= X_{k,\nt}^2 - Y_{k,\nt}^2 + 2i X_{k,\nt} Y_{k,\nt} \, . \] 
Taking conditional expectations both sides, we find
	\begin{equation}\label{fsquared}
	\frac{1}{2\pi} \int_{-\pi}^{\pi} \!\Big[ f(e^{-i \th } Z_{k,\nt} ) \Big]^2 \! \dd \th 
	= \bbE (X_{k,\nt}^2 | \mathscr{F}_{k+1,\nt}) - \bbE (Y_{k,\nt}^2 | \mathscr{F}_{k+1,\nt} ) + 2i 
	\bbE( X_{k,\nt} Y_{k,\nt} | \mathscr{F}_{k+1,\nt} ) . 
	\end{equation}
On the other hand, $f$ being holomorphic,  Cauchy's integral formula yields
	\[ \frac{1}{2\pi} \int_{-\pi}^{\pi} \Big[ f(e^{-i \th } Z_{k,\nt} ) \Big]^2 \dd \th  
	= \frac{1}{2\pi i} \int_{|z|=1} \bigg[ f\bigg(\frac{ Z_{k,\nt}}{z} \bigg) \bigg]^2 
	\frac{\dd z }{z}  = \lim_{|z|\to 0}    \bigg[ f\bigg(\frac{ Z_{k,\nt}}{z}  \bigg) \bigg]^2  
	= 0 \, . \]
Gong back to \eqref{fsquared}, this implies that both real and imaginary part of the r.h.s. must vanish almost surely, which is what we wanted to show. 
\end{proof}

Proposition \ref{simpler} already shows that the limiting Gaussian vector $\cF^\s (t,e^{ia})$ must have i.i.d. entries. It remains to compute the limiting variance, that is to show that for all $\eta >0$ it holds 
	\[\bbP \bigg( \bigg| \sum_{k=1}^\nt \bbE(X_{k,\nt}^2| \mathscr{F}_{k+1,\nt}) - v^2_t (\s) \bigg| >\eta \bigg)\to 0 \]
as $n\to \infty$. To this end it is clearly enough to work on $E(m,\e )$, the  advantage being that on this event we have 
	\[ \big| Z_{k,\nt}^\s (a)  - e^{ia + \s + (\nt-k)c} \big| \leq C(t) \e \]
for all $k\leq \nt$, as it follows directly from the definition of $E(m,\e )$ as long as $\s \gg \e$. 
Our strategy is then to replace each $Z_{k,\nt}$  by its deterministic approximation, and show that, provided $\s$ is large enough with respect to $c$, 
this does not affect the limiting variance. \\

Recall that the Poisson kernel for the unit disc $\bbD$ is given by $P_r (\th ) = \mathrm{Re} \big( \frac{1+re^{i\th}}{1-re^{i\th}} \big) $ for $r<1$, and that the function $re^{i\th} \mapsto P_r(\th)$ is harmonic in $\bbD$. Moreover, given any continuous function $f$ on $\bbT = \de \bbD$, its harmonic extension $Hf$ inside $\bbD$ is given by Poisson's integral formula 
	\[ (Hf) (re^{i\th}) = \frac{1}{2\pi} \int_{-\pi}^\pi P_r (\th - t) f (e^{it}) \dd t 
	= (P_r * f) (\th) \, . \]
We denote by $Q_r(\th)$ the harmonic conjugate of $P_r(\th)$ in $\bbD$, i.e. $Q_r (\th) =  \mathrm{Im} \big( \frac{1+re^{i\th}}{1-re^{i\th}} \big) $. 

\begin{Lemma}\label{LErep}
Assume $\s \gg \e$. Then
there exists a constant $C(t)$, depending only on $t$, such that on the event $E(m, \e )$ we have
	\[  \bigg| \bbE(X_{k,\nt}^2| \mathscr{F}_{k+1,\nt}) +c - \frac{c}{2\pi } \int_{-\pi}^\pi 
	\big(  P_{e^{-\s - (\nt-k)c}}(\th) \big) ^2 \dd \th \bigg| \leq \frac{C(t) c\e}{(\s + (\nt -k)c)^3} \, \]
for all $k\leq \nt$, and $n$ large enough.
\end{Lemma}
This result follows by a more general one, namely Lemma \ref{Llong} in the next section, and the proof is therefore omitted.  
Assume now that $\s \gg \sqrt{\e}$. Then we deduce from Lemma \ref{LErep} that on $E(m,\e )$ it holds: 
	\[ \begin{split} 
	\bigg| \sum_{k=1}^\nt \bbE(X_{k,\nt}^2| \mathscr{F}_{k,\nt}) - \sum_{k=1}^\nt \frac{c}{2\pi } \int_{-\pi}^\pi &\big(  P_{e^{-\s - (\nt-k)c}}(\th) \big) ^2 \dd \th + \nt c \bigg| 
	\leq  \sum_{k=1}^\nt  \frac{C(t) c\e }{(\s + (\nt -k)c)^3} \\ & 
	\leq C(t) \e \int_\s^{\s +\nt c} \frac{\dd x }{x^3} 
	=   C(t) \e \bigg( \frac{1}{2\s^2} - \frac{1}{2(\s+\nt c)^2} \bigg) \to 0 
	\end{split} \]
as $n \to \infty$. Note that $\sqrt{\e} \asymp \d^{1/3} $ (apart from logarithmic corrections), so the assumption $\s \gg \sqrt{\e}$ is stronger than the previous one $\s \gg \sqrt{\d}$. 

In conclusion, we have shown that that, provided $\s \gg \sqrt{\e}$, the limiting variance is given by the deterministic expression
	\[ \begin{split}
	\lim_{n\to\infty} \bigg( \frac{c}{2\pi} & \sum_{k=1}^\nt  \int_{-\pi}^\pi  
	\big(  P_{e^{-\s - (\nt-k)c}}(\th) \big) ^2\dd \th - \nt c \bigg)  
	=  \int_{\s}^{\s+t} \frac{1}{2\pi} \int_{-\pi}^{\pi} \big(  P_{e^{-x}}(\th) \big) ^2 \dd \th \dd x - t   \\ & 
	= 	\int_{\s}^{\s+t} (P_{e^{-x}}  * P_{e^{-x}})(0) \dd x -t 
	= 	\int_{\s}^{\s+t} \frac{1+e^{-2x}}{1-e^{-2x}} \dd x -t =  \log \frac{1-e^{-2(\s +t)}}{1-e^{-2\s}} 
	= v^2_t(\s )  \, , 
	\end{split} \]
where we have computed the inner integral by mean of Poisson's integral formula. Finally, if $\s >0$ is kept fixed as $n\to\infty$ the assumption $\s \gg \sqrt{\e}$ is trivially satisfied, so this concludes the proof of Theorem \ref{pointwise}.

\vspace{2mm}
\section{The fluctuation process on $C(\bbT)$} \label{S4}
Having a pointwise convergence result, it is natural to ask if this can be extended to obtain convergence of random fields. Recall that $\bbT = \{ |z| =1 \}$ denotes the unit circle, and let $C(\bbT)$ denote the space of continuous functions from $\bbT$ to $\bbC$, equipped with the supremum norm 
\begin{equation}\label{sup_norm} 
 \| x  \|_{\infty } = \sup_{\th \in [-\pi , \pi ) } | x (e^{i\th }) | \, . 
 \end{equation} 
Moreover, 
let $D[0,\infty )$ denote the space of c{\`a}dl{\`a}g functions $x:[0,\infty )\to C(\bbT )$. 
The goal of this section is to prove the following result.
\begin{Theorem} \label{Tfields}
Fix any $\s >0$,  and let $\cF^\s_n$ denote the $C(\bbT)$-valued c{\`a}dl{\`a}g stochastic process defined by 
	\[ \cF^\s _n (t,e^{ia}) = \frac{1}{\sqrt{c}} \bigg( \log \frac{\Phi_\nt (e^{ia+\s})}{e^{ia+\s}} -\nt c \bigg)  \]
for $e^{ia} \in \bbT$ and $t\geq 0$. 
Then there exists a continuous zero mean Gaussian process $\cF^\s : [0,\infty ) \to C(\bbT )$  whose covariance structure is given by 
	\[ 
	\mathrm{Cov} \big( \cF^\s (t, e^{ia}) \big) = v^2_t(\s)
	 \left( \begin{matrix} 1 & 0\\ 0 & 1 \end{matrix} \right) \, , \qquad 
	\mathrm{Cov }\big( \cF^\s (t,e^{ia}) ,  \cF^\s (s,e^{ib}) \big) = 
	\left( \begin{matrix} c_{s,t}(\s ,a-b) & \hat{c}_{s,t}(\s , a-b) \\ -\hat{c}_{s,t}(\s ,a-b ) & c_{s,t}(\s ,a-b) 
	\end{matrix}\right) \, ,
	\]
where $v^2_t(\s) = c_{t,t}(\s ,0)$, and for $s<t$
	\[ c_{s,t}(\s ,\a ) := \mathrm{Re} \bigg( \log  \frac{ 1-e^{-2\s -(t+s) +i\a } }{ 1 - e^{-2\s -(t-s) +i\a}} \bigg) \, , \qquad 
	\hat{c}_{s,t}(\s ,\a  ) = \mathrm{Im} \bigg( \log \frac{ 1-e^{-2\s -(t+s) +i\a } }{ 1 - e^{-2\s -(t-s) +i\a}} \bigg) \, ,  \]
such that $\cF^\s_n \to \cF^\s$ in distribution as $n\to\infty$, in the sense of weak convergence of probability measures on the space $D[0,\infty )$ equipped with the Skorokhod topology. 
\end{Theorem}
Note that the limiting process  is rotationally invariant in the spatial coordinate, as one expects from the rotation invariance of the original model. 
The rest of this section is devoted to the proof of the above result. \\

Fix any $\s >0$. It is trivial to check that $\cF_n^\s$ belongs to $D[0,\infty )$ for all $n\geq 0$. 
Since $D[0,\infty )$ equipped with the Skorokhod metric is a complete separable space, 
 it follows by Prohorov's theorem that $\cF^\s_n \to \cF^\s$ weakly if and only if the finite dimensional distributions (FDDs) of $\cF^\s_n$ converge to the ones of $\cF^\s$, and $(\cF_n^\s)_{n\geq 0}$ is tight (see \cite{ethier2009markov}, Lemma 4.3).

\subsection{Convergence of finite--dimensional distributions} \label{Sfdd}
The following result is a direct consequence of the discussion in \cite{iglehart1968weak}.
\begin{Lemma}[\cite{iglehart1968weak}] \label{L_FDD}
Assume that:
\begin{itemize}
\item[(i)] for any $t>0$ the family of probability measures of $(\cF^\s_n (t, \cdot ))_{n\geq 0}$ on $C(\bbT)$ is tight, and 
\item[(ii)] for any $0\leq t_1 < \ldots < t_M$ and any $-\pi \leq a_1 < \ldots < a_M <\pi $, $M \in \bbZ_+$,  it holds 
	\[ \left( \begin{matrix}
	\cF^\s_n (t_1, e^{ia_1} ) & \ldots & \cF^\s_n (t_M, e^{ia_1} ) \\
	\vdots &   & \vdots \\
	\cF^\s_n (t_1 , e^{ia_M} ) & \ldots & \cF^\s_n (t_M, e^{ia_M} ) 
	\end{matrix} \right) \longrightarrow
	\left( \begin{matrix}
	\cF^\s (t_1, e^{ia_1} ) & \ldots & \cF^\s (t_M, e^{ia_1} ) \\
	\vdots &   & \vdots \\
	\cF^\s (t_1, e^{ia_M} ) & \ldots & \cF^\s (t_M, e^{ia_M} ) 
	\end{matrix} \right)
	\] 
in distribution as $n\to\infty$.
\end{itemize}
Then the FDDs of $\cF^\s_n $ converge to the ones of $ \cF^\s$ as $n\to\infty$. 
More precisely, for any $M \in \bbN_+$ and any $0 \leq t_1 < \ldots < t_M$, the random vector $(\cF_n^\s (t_1 , \cdot ) , \ldots , \cF_n^\s (t_M , \cdot ) )$ in $C(\bbT)^M$ converges in distribution to $(\cF^\s (t_1 , \cdot ) , \ldots , \cF^\s (t_M , \cdot ) )$. 
\end{Lemma}
We start by showing that (ii) holds.  
To this end, let us again reduce to the scalar case by considering linear combinations. Recall the definition of the random variables $X_{k,n}^\s ( \, \cdot \, ) $, $Y_{k,n}^\s  ( \, \cdot \, ) $ given in \eqref{XY}. For $(\a_{lj})_{1\leq l,j\leq M}$ in $\bbR^{M\times M}$, we look at the weak limit of 
	\[ \sum_{l,j} \a_{lj} \cF_n^\s (t_l , e^{ia_j} ) =
	\sum_{k=1}^{\lfloor nt_M \rfloor } 
	\sum_{l,j} \a_{lj} \Big[ X_{k,\lfloor nt_l \rfloor }^\s (a_j) + iY_{k,\lfloor nt_l \rfloor}^\s (a_j) \Big]  \mathbf{1} (k\leq \lfloor nt_i \rfloor  ) 
	= \sum_{k=1}^{\lfloor nt_M \rfloor }  \mathscr{X}^\s_{k, \lfloor nt_M \rfloor },\]
where we have set 
	\begin{equation}\label{new_four}
	  \mathscr{X}^\s_{k ,\lfloor nt_M \rfloor }  = 
	\sum_{l,j} \a_{lj} \Big[ X_{k,\lfloor nt_l \rfloor }^\s (a_j) + iY_{k,\lfloor nt_l \rfloor}^\s (a_j) \Big]  \mathbf{1} (k\leq \lfloor nt_l \rfloor  )  
	\end{equation}
for $k \leq \lfloor nt_M \rfloor $. Then $\big( \mathscr{X}^\s_{k ,\lfloor nt_M \rfloor } \big)_{k\leq \lfloor nt_M \rfloor }$ is a backwards martingale difference array with respect to the filtration $(\mathscr{F}_{k,\lfloor nt_M \rfloor})_{k\leq \lfloor nt_M \rfloor }$. 
We can therefore apply Theorem \ref{ML_CLT} to show weak convergence, provided  that Assumptions (I) and (II) are satisfied. 
Note that all terms in the r.h.s. of \eqref{new_four} satisfy the estimates of Lemmas \ref{badE} and \ref{LEgoodE}, from which one can easily show, reasoning as in Lemma \ref{le3}, that Assumption (I) holds. Furthermore,  Lemma \ref{replace} is still in force, from which we conclude that, provided $\s \gg \sqrt{\d}$, the limiting variance is given by the limit in probability of $\ds \sum_{k=1}^{\lfloor nt_M \rfloor} \bbE \big( (  \mathscr{X}^\s_{k ,\lfloor nt_M \rfloor } )^2 | \mathscr{F}_{k+1,\lfloor nt_M \rfloor} \big)$. We now focus on the computation of this limit.  Expand the square and use linearity to see that the above sum equals 
	\[ \begin{split} 
	\sum_{l,j} \sum_{r,s} \a_{lj}\a_{rs} \sum_{k=1}^{\lfloor n t_l \rfloor \wedge
	\lfloor n t_r \rfloor} \bbE \bigg[ 
	\Big( X_{k,\lfloor nt_l\rfloor }^\s (a_j) + iY_{k,\lfloor nt_l \rfloor}^\s (a_j) \Big) 
	\Big( X_{k,\lfloor nt_r\rfloor }^\s (a_s) + iY_{k,\lfloor nt_r \rfloor}^\s (a_s) \Big) 
	\bigg| \mathscr{F}_{k, \lfloor n t_l \rfloor \vee \lfloor n t_r \rfloor} \bigg]
	\end{split}\]
It follows that it suffices to compute the limit in probability of 
	\[ \sum_{k=1}^\ns \bbE \Big[ 
	\Big( X_{k,\nt }^\s (a) + iY_{k,\nt}^\s (a) \Big) 
	\Big( X_{k,\ns }^\s (b) + iY_{k,\ns}^\s (b) \Big) 
	\Big| \mathscr{F}_{k, \nt } \Big] \]
for arbitrary $a,b \in [-\pi , \pi )$ and $0\leq s\leq t$. Moreover, by rotational invariance we can set $b=0$ without loss of generality.
The following result simplifies the computation.
\begin{Proposition}
Almost surely, it holds
	\[ \begin{split} 
	& \bbE (  X^\s_{k,\nt} (a) X^\s_{k,\ns}(0) | \mathscr{F}_{k+1,\nt} ) = 
	\bbE (  Y^\s_{k,\nt} (a)Y^\s_{k,\ns}(0) | \mathscr{F}_{k+1,\nt} ) \\ & 
	\bbE (  X^\s_{k,\nt} (a) Y^\s_{k,\ns}(0) | \mathscr{F}_{k+1,\nt} )  = 
	- \bbE (  Y^\s_{k,\ns} (0)X^\s_{k,\nt}(a) | \mathscr{F}_{k+1,\nt} ) 
	\end{split} \]
for all $k\leq \ns$.
\end{Proposition}
\begin{proof}
The result follows by the same arguments used in the proof of Proposition \ref{simpler}, considering now  $ \big[ f(e^{-i\Theta_k} Z^\s_{k,\nt}(a) ) + f(e^{-i\Theta_k} Z^\s_{k,\ns}(0)  ) \big]^2 $ in place of $\big[ f(e^{-i\Theta_k} Z_{k,\nt} ) \big]^2 $.
\end{proof}
It remains to compute the limit in probability of $ \sum_k\bbE (X^\s_{k,\nt}(a) X^\s_{k,\ns}(0) | \mathscr{F}_{k+1,\nt})$ and\\ $\sum_k \bbE (X^\s_{k,\nt}(a) Y^\s_{k,\ns}(0) | \mathscr{F}_{k+1,\nt})$. As in the previous section, we do this by approximating by a deterministic quantity. 

\begin{Lemma}\label{Llong} 
Assume $\s \gg \e$. Then there exists a constant $C(t)$, depending only on $t$ ($t>s$), such that on the event $E(m,\e )$ the following hold:
\[ \begin{split}
	 &\bigg| \bbE(X^\s_{k,\nt}(a)X^\s_{k,\ns}(0)| \mathscr{F}_{k+1,\nt})  + c 
	 - \frac{c}{2\pi } \int_{-\pi}^\pi \!\!
	P_{e^{-\s - (\nt -k)c}}(a-\th) P_{e^{-\s - (\ns -k)c}}(\th) \dd \th  \bigg| \leq 
	\frac{C(t) c\e}{(\s \! +\! (\ns -k)c)^3} \,  \\ & 
	\bigg| \bbE(X^\s_{k,\nt}(a)Y^\s_{k,\ns}(0)| \mathscr{F}_{k+1,\nt })  + c 
	- \frac{c}{2\pi } \int_{-\pi}^\pi \!\!
	P_{e^{-\s - (\nt -k)c}}(a-\th) Q_{e^{-\s - (\ns -k)c}}(\th) \dd \th  \bigg| \leq 
	\frac{C(t) c\e}{(\s \! +\! (\ns -k)c)^3} 
	\end{split} \]
for all $k\leq \ns$ and $n$ large enough. Above $Q_r(\th) = \mathrm{Im} \big( \frac{1+re^{i\th}}{1-re^{i\th}} \big)$, $r<1$, denotes the conjugate Poisson kernel.
\end{Lemma}
We discuss the proof of Lemma \ref{Llong}, from which Lemma \ref{LErep} also follows by setting $a=0$ and $s=t$, in Appendix \ref{Alemma}.
As a consequence of the above result we obtain that, after further assuming $\s \gg \sqrt{\e}$, on the event $E(m,\e)$ it holds:
	\[ \begin{split}
	 \bigg| \sum_{k=1}^\ns \bbE(X^\s_{k,\nt}(a) & X^\s_{k,\ns}(0)| \mathscr{F}_{k+1,\nt})  + \ns c  - \frac{c}{2\pi } \sum_{k=1}^\ns \int_{-\pi}^\pi 
	P_{e^{-\s - (\nt -k)c}}(a-\th) P_{e^{-\s - (\ns -k)c}}(\th) \dd \th  \bigg| \leq \\ & 
	\leq  \sum_{k=1}^\ns \frac{C(t) c\e }{(\s + (\ns -k)c)^3}
	 \leq C(t) \e \int_\s^{\s +\ns c} \frac{\dd x }{x^3} 
	=   C(t) \e \bigg( \frac{1}{2\s^2} - \frac{1}{2(\s+\ns c)^2} \bigg) \to 0 
	\end{split} \]
as $n \to \infty$. This in turn implies  that 
	\[ \begin{split} 
	\lim_{n\to \infty } &\sum_{k=1}^\ns \bbE(X^\s_{k,\nt }(a)  X^\s_{k,\ns }(0)| \mathscr{F}_{k+1,\nt}) 
	= \lim_{n\to\infty} \bigg( \frac{c}{2\pi } \sum_{k=1}^\ns \int_{-\pi}^\pi 
	P_{e^{-\s - (\nt -k)c}}(a-\th) P_{e^{-\s - (\ns -k)c}}(\th) \dd \th - \ns c \bigg) \\ 
	& = \frac{1}{2\pi} \int_\s^{\s +s} \!\!\int_{-\pi}^\pi P_{e^{-(t-s) -x}}(a-\th) P_{e^{-x}}(\th )\dd \th -s 
	= 	\int_{\s + \frac{t-s}{2}}^{\s+\frac{t+s}{2} } P_{e^{-2x}} (a)\dd x -s =  
	\log \bigg| \frac{ 1-e^{-2\s -(t+s) +ia } }{ 1 - e^{-2\s -(t-s) +ia}} \bigg| \, , 
	\end{split} \]
as claimed. Similarly, 
	\[ \begin{split} 
	\lim_{n\to \infty } \sum_{k=1}^\ns \bbE(X^\s_{k,\nt}(a) Y^\s_{k,\ns}(0)|  
	\mathscr{F}_{k+1,\nt}) & 
	= \frac{1}{2\pi} \int_\s^{\s +s} \!\!\int_{-\pi}^\pi P_{e^{-(t-s) - x}}(a-\th) Q_{e^{-x}}(\th )\dd \th 
	\\ & = 		\int_{\s + \frac{t-s}{2}}^{\s+\frac{t+s}{2} }  Q_{e^{-2x}} (a)\dd x  = 
	\textrm{Arg} \bigg( \frac{ 1-e^{-2\s -(t+s) +ia } }{ 1 - e^{-2\s -(t-s) +ia}} \bigg) \, . 
	\end{split} \]
This concludes the proof of (ii) in Lemma \ref{L_FDD}. It remains to show that (i) holds. 

Pick any $t>0$. Suppose we could show that:
\begin{itemize}
\item[(a)] for all $\nu >0$, there exists $M_\nu >0$ and $n_0 \in \bbN_+$ such that 
$\ds  \sup_{n\geq n_0} \bbP \big( \| \cF^\s_n (0, \cdot ) \|_\infty >M_\nu \big) \leq \nu $, and
\item[(b)] for all $\nu >0$ it holds 
	\begin{equation}\label{tight}
	\lim_{\eta \to 0} \limsup_{n\to\infty} \bbP \bigg(
	\sup_{|a-b|<\eta} | \cF_n^\s (t , e^{ia}) - \cF_n^\s (t, e^{ib}) | \geq \nu \bigg) =0 \, . 
	\end{equation}	
\end{itemize}
Then it would follow from Theorem 7.5 in \cite{billingsley1999convergence} that the sequence of probability measures on $C(\bbT)$ associated to $(\cF^\s_n (t , \cdot ) )_{n\geq 0}$ is tight.

Since $\cF^\s_n (0, \cdot ) \equiv 0$, (a) is trivially satisfied. The fact that also (b) holds follows by the same reasoning as in Section \ref{Stight} below, and the proof is therefore omitted.
	
\subsection{Tightness} \label{Stight}
To conclude the proof of Theorem \ref{Tfields} it remains to show that the family of probability measures associated to $(\cF^\s_n )_{n\geq 0}$ on $D[0,\infty )$ is tight.
Corollary 7.4 in \cite{ethier2009markov}  
provides a characterization for tightness in this space. In particular, it tells us that if:
\begin{itemize}
\item[(a')]  for all $\nu >0$ and all $ t\in [0,\infty ) \cap \bbQ$, there exists a positive constant $M = M(\nu, t) $  such that 
	\[ \limsup_{n\to\infty} \bbP ( \| \cF^\s_n(t , \cdot ) \|_\infty > M ) <\nu\, , \]
and
\item[(b')] for all $\nu,T>0$ there exists a positive constant $\eta = \eta (\nu , T)$ such that 
	\[ \limsup_{n\to\infty} \bbP (\omega  (\cF^\s_n , \eta , T) \geq \nu ) \leq \nu\, , \]	
\end{itemize}
where $\omega $ denotes the 
modulus of continuity on $D[0,\infty )$, 
then the desired tightness follows. We start by showing (b'). 
 It is clear that it suffices to restrict to the event $E(m,\e )$. We are set to show that for any $\nu , T>0$ it holds 
	\begin{equation} \label{big_tight}
	\limsup_{n\to \infty} \bbP \Bigg( 	 \sup_{\substack{s,t \in [0,T],|t-s|<\eta \\ 
	\a \in [-\pi , \pi )}} | \cF_n^\s (t,e^{i\a} ) - \cF_n^\s (s,e^{i\a} ) | \geq \nu \, ; \,  E(m,\e ) \Bigg) \to 0 \qquad \mbox{ as } \eta \to 0\, . 
	\end{equation}
It is convenient to switch to logarithmic coordinates. 
Following the notation introduced in \cite{norris2012hastings}, Section 5, we set $\tilde{D}_0 = \{ w \in \bbC : \mathrm{Re}(w)>0 \} $, $\tilde{D} = \{ w \in \bbC : e^w \in D \}$, and let $\tilde{F} $ be the unique conformal map from $\tilde{D}_0$ to $\tilde{D}$ such that $\tilde{F}(w) = w+c+o(1)$ as $\mathrm{Re}(w)\to\infty$. Moreover, let $\tilde{G} = \tilde{F}^{-1}$, so that $\tilde{G}(w) = w-c+o(1)$ as $\mathrm{Re}(w)\to\infty$. Finally, for all $k\leq n$ set $\tilde{F}_k (w)= \tilde{F}(w-i\Theta_k ) + i\Theta_k $, $\tilde{G}_k = \tilde{F}_k^{-1}$ and $\tilde{\Phi}_k = \tilde{F}_1 \circ \cdots \circ \tilde{F}_k$, $\tilde{\Gamma}_k = \tilde{\Phi}_k^{-1}$. Then $\cF^\s_n (t, \th ) = \frac{1}{\sqrt{c}} \Big( \tilde{\Phi}_{\nt } (i\th + \s ) - i\th - \s -\nt c \Big)$ and, assuming $t\geq s$ without loss of generality, we have:
	\begin{equation} \label{events2}
	\begin{split}
	\bbP \Bigg( & \sup_{\substack{s,t\in[0,T], |t-s|<\eta \\ \a \in [-\pi ,\pi )}}  |  \cF_n^\s (t, e^{i\a} ) -  \cF_n^\s (s,e^{i\a} ) | \geq  \nu \,  ; E(m,\e )\Bigg) = \\
	& = \bbP \Bigg( \sup_{\substack{ |t-s|<\eta \\ z: \Re (z) =\s }} 
	\bigg| \Big( \tilde{\Phi}_{\nt} (z) - z -\nt c \Big) - 
	\Big( \tilde{\Phi}_{\ns} (z ) - z -\ns c \Big) \bigg| \geq \nu \sqrt{c} \, ; \,
	E(m,\e ) \Bigg)  \, . 
	\end{split}
	\end{equation}
Note that on $E(m,\e )$ we have $z = \tilde{\Gamma}_{\nt} (w) = \tilde{\Gamma}_{\ns}(w')$ for some $w\in \tilde{D}_{\nt}$ and $w' \in \tilde{D}_{\ns}$ such that $|\Re (w) - \s - \nt c |< \e $, $|\Re (w') - \s - \ns c |<\e $. Moreover, since $|t-s|<\eta$, it must be   $|w-w'|<\nt - \ns  +4\e < 4\eta$ as long as $n$ is large enough. 
For any $k \geq 0$, if $w \in \tilde{D}_k$ define  $M_k(w) =  \tilde{\Gamma}_k (w ) - w +kc$. Finally, let 
	\[ \mathcal{T} := \inf \big\{ k\geq 0 : \xi \notin \tilde{D}_k \mbox{ for some }\xi \mbox{ such that } \Re (\xi ) = \s + k -\e \big\}, \]
and note that on $E(m,\e)$ we have $\mathcal{T}> \nt $ for all $t\leq T$. 
It follows that  the r.h.s. of  \eqref{events2} is bounded above by
	\[ \begin{split}
	 \bbP \Bigg( \sup_{\substack{ |t-s|<\eta \\ w\, : |\Re (w) -\s - \nt c|<\e \\  w': |\Re (w') -\s - \ns c|<\e \\ |w-w'|<4\eta }} 
	\bigg| \Big(  \tilde{\Gamma}_{\nt} & (w) - w +\nt c  \Big) - 
	\Big( \tilde{\Gamma}_{\ns} (w' ) - w' +\ns c  \Big) \bigg| \geq \nu \sqrt{c} \, ; \, E(m,\e ) \Bigg)  \\
	& \leq 
	\bbP \Bigg( \sup_{\substack{ |t-s|<\eta \\ w\, : |\Re (w) -\s - \nt c|<\e \\ 
	 w': |\Re (w') -\s - \ns c |<\e \\ |w-w'|<4\eta }} 
	 \Big| M_{\nt \wedge \mathcal{T}} (w) - M_{\ns \wedge \mathcal{T}} (w') \Big|
	 \geq \nu \sqrt{c}\Bigg) .
	  \end{split} \]
We control the above probability by mean of a $3$-dimensional version of Kolmogorov's continuity theorem (cf. \cite{durrett1996stochastic}, Theorem 1.6). To this end, we show the following.

\begin{Lemma}\label{Lnuovo}
Fix any $s,t \in [0,T]$ with $s\leq t$, and any $w, w'$ such that $ | \Re (w) -\s - \nt c |<\e$, $|\Re (w') -\s - \ns c |<\e$ and $|w-w'|< 1$. Then there exists a constant $C=C(\s , T)$, depending only on $\s$ and $T$, such that 
	\begin{equation*} 
	 \bbE \bigg( \Big| M_{\nt \wedge \mathcal{T}} (w) - M_{\ns \wedge \mathcal{T}} (w') \Big|^8 \bigg) \leq C(\s , T) c^4 | (t,w) - (s,w') |^4  ,
	 \end{equation*}
where $|(t,w) - (s,w')| $ denotes the Euclidean norm in $\bbR^3$. 
\end{Lemma}

\begin{proof}
Clearly the l.h.s. of the above inequality  is upper bounded by 
	\begin{equation}\label{newgoal}
	C \bigg[ 
	\bbE \bigg( \Big| M_{\nt \wedge \mathcal{T}} (w) - M_{\ns \wedge \mathcal{T}} (w) \Big|^8 \bigg) + 
	 \bbE \bigg( \Big| M_{\ns \wedge \mathcal{T}} (w) - M_{\ns \wedge \mathcal{T}} (w') \Big|^8 \bigg) \bigg] 
	 \end{equation}
for some absolute constant $C>0$. We control the two terms separately. 

Set $\tilde{G}_0(w) =\tilde{G}(w) - w$, so that $M_{k+1}(w) - M_k(w) = \tilde{G}_0(\tilde{\Gamma}_k (w)-i\Theta_{k+1}) +c$, and 
	\[ \bbE (M_{k+1}(w) - M_k(w)  | \s (\Theta_1 , \ldots , \Theta_k) ) =
	\frac{1}{2\pi} \int_{-\pi}^\pi \tilde{G}_0(\tilde{\Gamma}_k (w) -i\th ) \dd \th  + c
	= 0\]
for all $k\leq \nt \wedge \mathcal{T}$. It follows that $\big( M_{k\wedge \mathcal{T}}(w) \big)_{k\leq \nt}$ is a martingale. Moreover, we deduce from the estimates in \eqref{tNT}, which now read 
	\begin{equation}\label{log_bounds}
	 | \tilde{G}_0(w) + c | \leq \frac{Cc}{\mathrm{Re} (w) -\d } \, , 
	\qquad \quad  | \tilde{G}_0' (w) | \leq \frac{Cc}{(\textrm{Re} (w) -\d )^2 } \, ,
	\end{equation}
that $| M_{(k+1)\wedge \mathcal{T} }(w) - M_{ k\wedge \mathcal{T} }(w) | = 
| \tilde{G}_0 ( \tilde{\Gamma}_k(w) -i\Theta_{k+1})+c| \leq Cc/\s$ for all $k < \nt$.
Using orthogonality of the increments of $(M_{k\wedge \mathcal{T}})$, then, we find
	\begin{equation}\label{newsecond}
	\begin{split} 
	\bbE \bigg( \Big| M_{\nt \wedge \mathcal{T}} &(w) - 
	M_{\ns \wedge \mathcal{T}} (w) \Big|^8 \bigg)  =
	\bbE \bigg( \Big| \sum_{k=\ns}^{\nt -1} \big( M_{(k+1) \wedge \mathcal{T}} (w) - 
	M_{k \wedge \mathcal{T}} (w) \big) \Big|^8 \bigg) 
	\\ & \leq C (\nt - \ns )^3  \sum_{k=\ns}^{\nt -1} 
	\bbE \Big( \big| M_{(k+1) \wedge \mathcal{T}} (w) - 
	M_{k \wedge \mathcal{T}} (w) \big|^8 \Big) 
	\leq \frac{C c^4}{\s^8} |t-s|^4 \, 
	\end{split} 
	\end{equation}
for some absolute constant $C>0$ and $n$ large enough. \\

Let us now look at the second term in \eqref{newgoal}. For $k\leq \ns \wedge \mathcal{T}$ set $\tilde{M}_k = M_k(w) - M_k(w')$. Then $\big( \tilde{M}_{k\wedge \mathcal{T}} \big)_{k\leq \nt}$ is a martingale, and by \eqref{log_bounds}
	\[ \big| \tilde{M}_{(k+1)\wedge \mathcal{T}} - \tilde{M}_{k\wedge \mathcal{T}} \big|
	= \big|  \tilde{G}_0(\tilde{\Gamma}_k (w-i\Theta_{k+1}) -
	 \tilde{G}_0(\tilde{\Gamma}_k (w' -i\Theta_{k+1}) \big| 
	 \leq \frac{Cc}{\s^2} ( |w-w'| + | \tilde{M}_{k\wedge \mathcal{T}} | )\, . \]
This, together with orthogonality of the increments, yields
	\[ \begin{split}
	\bbE \bigg( \Big| \tilde{M}_{\ns \wedge \mathcal{T}}  \Big|^8 \bigg) &
	= \bbE \bigg( \Big| \sum_{k=0}^{\ns -1} \big( \tilde{M}_{(k+1) \wedge \mathcal{T}}  - \tilde{M}_{k \wedge \mathcal{T}}  \big) \Big|^8 \bigg) 
	 \leq \ns \sum_{k=0}^{\ns -1} \bbE \Big( \big| \tilde{M}_{(k+1) \wedge \mathcal{T}}  - \tilde{M}_{k \wedge \mathcal{T}}  \big|^8 \Big) \\ &
	\leq \ns ^3 \frac{Cc^8}{\s^{16}} \sum_{k=0}^{\ns -1} 
	\Big( |w-w'|^8 + \bbE \big( |\tilde{M}_{k \wedge \mathcal{T}} |^8 \big) \Big) \, .
	\end{split} \]
It then follows from Gr\"{o}nwall's inequality that 
	\begin{equation}\label{newfirst}
	 \bbE \bigg( \Big| \tilde{M}_{\ns \wedge \mathcal{T}}  \Big|^8 \bigg)  
	\leq \ns^4 \frac{Cc^8}{\s^{16}} |w-w'|^8 e^{\frac{C\ns^4 c^8}{\s^{16}}} 
	\leq C(\s , T) c^4 |w-w'|^8 
	\end{equation}
for $n$ large enough and a constant $C(\s , T)$ depending only on $\s$ and $T$.
Putting \eqref{newsecond} and \eqref{newfirst} together, and using that $|w-w'|^8 \leq |w-w'|^4$ since $|w-w'|<1$ by assumption, we find
	\[ \begin{split} 
	\bbE \bigg( \Big| M_{\nt \wedge \mathcal{T}} &(w) - 
	M_{\ns \wedge \mathcal{T}} (w) \Big|^8 \bigg)  + 
	\bbE \bigg( \Big| \tilde{M}_{\ns \wedge \mathcal{T}}  \Big|^8 \bigg)   
	\leq C'(\s , T) c^4 |(t,w) - (s,w')|^4 \, ,
	\end{split}\]
for $n$ large enough and $C'(\s , T) = \max \{ C/\s^8 ; C(\s , T) \}$, as claimed.
\end{proof}
Kolmogorov's continuity theorem now yields the existence of a random variable $\mathscr{M}>0$ such that 
	\[ \sup_{\substack{ |t-s|<\eta \\ w\, : |\Re (w) -\s -\nt c |<\e \\ 
	 w': | \Re (w') -\s - \ns c |<\e \\ |w-w'|<4\eta }} 
	 \Big| M_{\nt \wedge \mathcal{T}} (w) - M_{\ns \wedge \mathcal{T}} (w') \Big|
	 \leq \mathscr{M} |(t,w) - (s,w')|^{1/16} \, , \]
with $\bbE (\mathscr{M}^8) \leq C(\s , T) c^4$, for $C(\s , T)$ as in the statement of Lemma \ref{Lnuovo}. Therefore we find
	\[ \bbP \Bigg( \sup_{\substack{ |t-s|<\eta \\ w\, : | \Re (w) -\s - \nt c |<\e \\ 
	 w': | \Re (w') -\s - \ns c |<\e \\ |w-w'|<4\eta }} 
	 \Big| M_{\nt \wedge \mathcal{T}} (w) - M_{\ns \wedge \mathcal{T}} (w') \Big|
	 \geq \nu \sqrt{c}\Bigg)  \leq 
	 \bbP \Big( \mathscr{M}^8 \geq \frac{\nu^8 c^4}{4\sqrt{\eta}} \Big) 
	 \leq 4 C(\s , T) \frac{\sqrt{\eta}}{\nu^8} \, , \] 
and sending first $n\to\infty$ and then $\eta \to 0$ we conclude that \eqref{big_tight} holds. 
This proves tightness, and hence it concludes the proof of Theorem \ref{Tfields}.

\vspace{3mm}

\section{The fluctuation process on $\mathscr{H}$} \label{Shol}
\textbf{Notation.}
Recall that $\bbT = \{ z\in \bbC : |z|=1 \}$, and set $\a \bbT = \{ z \in \bbC : |z| =\a \}$ for any $\a \in \bbR_+$. Let $\big( C(\a \bbT ) , \| \cdot \|_\infty \big)$ denote the space of continuous function on $\a \bbT$ equipped with the supremum norm. 
Moreover, for a subset $D$ of the complex plane, denote by $\mathscr{H}(D)$ the space of holomorphic functions on $D$. Whenever $D = \{ |z| >1 \}$, denote $\mathscr{H}(D)$ simply by $\mathscr{H}$.\\

We have shown in the previous section that, for any fixed $\s >0$, $\cF^\s_n \to \cF^\s$ as $n\to \infty$ in distribution with respect to the Skorokhod topology on the space $D[0,\infty )$ of c{\`a}dl{\`a}g  functions from $[0,\infty )$ to $C(\bbT )$. 
Note that, for any $t\geq 0$, the continuous function $\cF^\s_n (t, \cdot )$ coincides with the restriction of the holomorphic function 
	\begin{equation}\label{bigFn}
	\cF_n (t,z) = \frac{1}{\sqrt{c} } \Big( \log \frac{\Phi_\nt (z)}{z} -\nt c \Big) \, , 
	\end{equation}
defined for all $|z|>1$, to the circle $e^\s \bbT $. 
We show below that also the limit $\cF^\s (t,\cdot )$ can be interpreted as the restriction of a holomorphic random function $\cF(t,\cdot )$, defined for all $|z|>1$, to $e^\s \bbT$. Moreover, we provide an explicit construction of $\cF$, and prove that $\cF_n \to \cF$ in distribution as $n\to\infty$, 
with respect to the Skorokhod topology on the space of c{\`a}dl{\`a}g  functions from $[0,\infty )$ to $\mathscr{H}$. \\

Define, for all $N\geq 1$, a distance $d_N$ on $\mathscr{H}$ by setting 
\begin{equation}\label{distances}
 d_N(\phi , \psi ) = \sup_{|z| \geq e^{1/N} } |\phi (z) - \psi (z) |  \wedge 1 
\, ,  \quad  \mbox{ and let } \quad 
d (\phi , \psi ) = \sum_{N\geq 1 } \frac{d_N(\phi , \psi ) }{2^N}\, . 
\end{equation}
Since $(\mathscr{H},d_N)$ is a complete separable metric space for all $N\geq 1$, this makes $(\mathscr{H}, d)$ into a complete separable metric space (i.e. Polish space). 
It follows that, if $D_\mathscr{H}[0,\infty )$ denotes the space of c{\`a}dl{\`a}g  functions from $[0,\infty )$ to $\mathscr{H}$ equipped with the Skorokhod metric $\mathbbm{d}_\mathscr{H}$, then also $(D_\mathscr{H}[0,\infty ) , \mathbbm{d}_\mathscr{H})$ is a complete separable metric space. 
It is clear that each $\cF_n$ defined in \eqref{bigFn} is a random variable in $D_\mathscr{H}[0,\infty )$. 

Let us now describe the explicit construction of $\cF$. 
Let $D = \bbR / 2\pi \bbZ$, and recall that if we set $ e_k(\th) = e^{ik\th} /\sqrt{2\pi}$ for $k\in \bbZ$, then 
$(e_k)_{k\in\bbZ} $ 
forms an orthonormal basis (in short ONB)  for $L^2(D)$ with respect to the inner product $(f,g) = \int_{-\pi}^\pi \overline{f(\th )} g(\th) \dd \th$. 
On this basis, that we refer to as Fourier basis, the Poisson kernel $\Re \big(\frac{1+z}{1-z}\big)$ reads
	\[ P_{1/r} (\th ) = \Re \Big( \frac{1+e^{i\th}/r}{1-e^{i\th}/r} \Big) 
	= \sqrt{2\pi} \Big( e_0 + \sum_{k \in \bbZ \setminus \{ 0\}}   r^{-|k|} e_k(\th ) \Big) \]
for any  $r>1$. 
Take two independent collections $(\b_k)_{k\in\bbZ}$ and $(\b'_k)_{k\in \bbZ}$ of i.i.d. Brownian Motions on $\bbR$, and denote by $(A_{k})_{k\in \bbZ}$ and $(B_{k})_{k\in \bbZ}$ the solutions to 
	\[ \begin{cases}
	\dd A_k (t) = -|k| A_k(t) \dd t + \sqrt{2} \,\dd \b_k(t) \, , \\
	A_k(0)=0 
	\end{cases}
	\qquad 
	\begin{cases}
	\dd B_k (t) = -|k| B_k(t) \dd t + \sqrt{2} \,\dd \b'_k(t)  \\
	B_k(0)=0 \,.
	\end{cases}\]
Then $A_k , B_k$ perform independent Ornstein-Uhlenbeck processes on $\bbR$ with invariant distribution $\cN (0,1/|k|)$.
Define formally
	\begin{equation} \label{Wformal} 
	\cW (t,\th) = \sum_{k\in \bbZ \setminus \{ 0 \} } \Big( \frac{A_{k}(t) + i B_{k}(t)}{\sqrt{2}}\Big)  e_k(\th) 
	\end{equation}
for $(t,\th ) \in [0,\infty ) \times [-\pi,\pi )$. Finally, for $(t, re^{ia }) \in [0,\infty ) \times \{|z|>1\}$ set 
	\[ \cF (t,re^{ia}) = \frac{1}{\sqrt{2\pi}} \Big(P_{1/r}  * \cW(t,\cdot ) \Big) (a) 
	= \frac{1}{\sqrt{2\pi}}  \int_{-\pi }^\pi \Re \Big( \frac{1+e^{i (a-\th )}/r}{1-e^{i (a-\th )}/r} \Big) \,  \cW(t,\th  ) \dd \th \, .   \]

\begin{Theorem}\label{Thol_new}
$\cF$ is a random variable in $C_\mathscr{H} [0,\infty ) \subset D_\mathscr{H}[0,\infty )$. Moreover, $\cF_n \to \cF$ as $n\to\infty$ in distribution with respect to  the Skorokhod metric $\mathbbm{d}_\mathscr{H}$ on $D_\mathscr{H}[0,\infty )$. 
\end{Theorem}

The question of how to make sense of the boundary values $\cW$ defined formally in \eqref{Wformal} is addressed in the next section, where we show that $\cW$ can be rigorously defined as an Ornstein-Uhlenbeck process on a suitable infinite-dimensional Hilbert space.

The rest of this section is devoted to the proof of Theorem \ref{Thol_new}. 

\begin{Lemma}\label{lemma_new}
$\cF$ is a random variable in $C_\mathscr{H}[0,\infty )$. Moreover, $\cF$ is Gaussian, and its restriction to $ [0,\infty ) \times e^\s \bbT$ agrees in distribution with $\cF^\s$ defined in Theorem \ref{Tfields}, for any $\s >0$.
\end{Lemma}
\begin{proof}
The fact that $\cF$ is Gaussian is true by construction. 
Fix any $t\geq 0$, and expand $\cF (t,\cdot )$ in Fourier basis, to get
	\begin{equation}\label{expand}
	 \begin{split} 
	\cF (t, re^{ia}) & \!= 
	\bigg(  e_0 + \sum_{k\neq 0} r^{-|k|} e^{-ika}e_{k}  , \,
	\sum_{k\neq 0} \Big( \frac{A_{k}(t) + i B_{k}(t)}{\sqrt{2}}\Big)  e_k \bigg)  =
	 \sum_{k\neq 0}  r^{-|k|} \Big( \frac{ A_{k}(t) +i B_{k}(t) }{\sqrt{2}} \Big)  e^{ika}
\\ & \stackrel{(d)}{=}
 \sum_{k\geq 1}  r^{-k} \big[ A_{k}(t) \cos ka -B_{k}(t) \sin ka \big]
	+ i \sum_{k\geq 1}  r^{-k} \big[ B_{k}(t) \cos ka + A_{k}(t) \sin ka \big]
	\end{split} 
	\end{equation}
for any $r>1$ and $a \in [-\pi,\pi )$, where the last equality holds in law as stochastic processes, and it follows from the independence of the OU processes. This provides an almost surely convergent power series expansion for $\cF(t,\cdot )$ at all $z$ with $|z|>1$, and hence it shows that $\cF$ is a Gaussian stochastic processes taking values in $\mathscr{H}$.
Recall from \eqref{distances} the definition of the distances $(d_N), d$ on $\mathscr{H}$.
To see that $\cF$ is continuous, we have to show that for all compacts of the form $[0,T]$ it holds almost surely that, if $t_n \to t \in [0,T]$ as $n\to\infty$, then $d\big(\cF(t_n , \cdot ) , \cF(t, \cdot ) \big) \to 0$, i.e. $d_N\big(\cF(t_n , \cdot ) , \cF(t, \cdot ) \big) \to 0$ for all $N\geq 1$. We have:
	\[ \begin{split} 
	d_N\big(\cF(t_n , \cdot ) , \cF(t, \cdot ) \big) & =
	\sup_{|z| \geq e^{1/N} }\big| \cF(t_n , z ) - \cF(t,z ) \big|
	 = \sup_{|z| = e^{1/N}} \big| \cF(t_n , z ) - \cF(t,z ) \big|
	 \\ &  = \sup_{a\in [-\pi , \pi)} \Big| 
	 \sum_{k\geq 1} e^{-k/N} \big[ ( A_{k}(t_n) +i B_{k}(t_n)  )  -  ( A_{k}(t) +i B_{k}(t)  )  
	 \big] e^{ika} \Big| \\
	 & \leq \,  \sum_{k\geq 1} e^{-k/N} \Big( | A_{k}(t_n) - A_{k}(t) | 
	 +   | B_{k}(t_n) - B_{k}(t) | \Big) \, . 
	\end{split} \]
To show that the last term converges to $0$ as $n\to\infty$ almost surely for all $t_n , t\in [0,T]$, then, it suffices to prove that the sequence of continuous functions 
$ g_M(t) = \sum_{k= 1}^M e^{-k/N}  | A_{k}(t)  | $ converges uniformly on $[0,T]$ as $M\to\infty$. Indeed, we find
	\[ 
	\sup_{t\in [0,T] } | g_M(t) - g_\infty (t) | 
	= \sup_{t\in [0,T]} \sum_{k=M+1}^\infty  e^{-k/N} |A_k(t)|
	\leq \sum_{k=M+1}^\infty e^{-k/N}  \sup_{t\in [0,T]}  |A_k(t)| \, . \]
On the other hand Doob's maximal inequality for the submartingale $( e^{-2kt} A_k^2(t) )_{t\leq T}$ yields
	\[ \bbP \Big( \sup_{t\in [0,T]}  |A_k(t)| \geq e^{k/2N} \Big) 
	\leq \bbP \Big( \sup_{t\in [0,T]} \{ e^{-2kt} A_k^2(t) \} \geq e^{k/N - 2kT} \Big) 
	\leq \frac{\bbE (A_k^2(T))}{e^{k/N}} \leq \frac{1}{k e^{k/N}} \, ,  \, \]
so by Borel-Cantelli  $\ds \sup_{t\in [0,T]} |A_k(t)| < e^{k/2N} $ for almost all $k\geq 1$, almost surely. This proves uniform convergence on compacts, and hence almost sure continuity of $\cF$. \\

To conclude the proof, we show that $\cF$ has the same covariance structure of $\cF^\s$ on the circle $e^\s \bbT$, for arbitrary $\s >0$. Indeed, it follows from \eqref{expand} that real and imaginary parts of $\cF (t,e^{\s +ia})$, $a\in [-\pi , \pi )$, are independent centred real Gaussian random variables, with
	\[ \bbE \big[ (\mathrm{Re} \cF(t,e^{\s+ia} ))^2 \big]
	= \bbE \big[ (\mathrm{Im} \cF(t,re^{\s+ia} ) )^2\big] 
	=  \sum_{k\geq 1} \frac{e^{-2k\s} (1-e^{-2kt})}{k} = 
	\log \frac{ 1- e^{-2(t+\s)}}{1-e^{-2\s}} \, . \]
Moreover, expanding $\cF (s, e^\s )$ as in \eqref{expand}, one further checks that for $s<t$
	\[ \mathrm{Cov} \big( \mathrm{Re} \cF (t,e^{\s+ia} ) , 
	\mathrm{Re} \cF (s,e^\s ) \big) 
	= \mathrm{Cov}\big(\mathrm{Im} \cF (t,e^{\s+ia} ) , \mathrm{Im} \cF (s,e^\s) \big) 
	= \mathrm{Re}\bigg(\!\log \frac{ 1- e^{-(t+s)-2\s +ia}}{1-e^{-(t-s)-2\s+ia}}\bigg) , \]
and 
	\[ \mathrm{Cov}\big(\mathrm{Re} \cF (t,e^{\s+ia} ),\mathrm{Im} \cF (s,e^\s ) \big) 
	= - \mathrm{Cov}\big(\mathrm{Im}\cF (t,e^{\s+ia} ),\mathrm{Re} \cF (s,e^\s ) \big) 
	= \mathrm{Im} \bigg( \log \frac{ 1- e^{-(t+s)-2\s+ia}}{1-e^{-(t-s)-2\s+ia}}  \bigg). \]
By rotational invariance in the spatial coordinate, this is enough to conclude that $ \cF $ and $\cF^\s$ have the same covariance structure, and hence the same law, on every circle of the form $e^\s \bbT$, for arbitrary $\s >0$, as claimed. 
\end{proof}

\begin{proof}[Proof of Theorem \ref{Thol_new}]
For any $N\geq 1$, denote by $H_N$ the operator that maps a continuous function on $e^{1/N} \bbT$ to its holomorphic extension to the outer region $\{ |z| \geq e^{1/N} \}$. 
Moreover, let $D_N[0,\infty )$ and $D_{\mathscr{H}_N} [0,\infty )$ denote respectively the space of c{\`a}dl{\`a}g functions from $[0,\infty ) $ to $\big( C(e^{1/N} \bbT ), \| \cdot \|_\infty \big) $,  and the space of c{\`a}dl{\`a}g functions from $[0,\infty ) $ to $\big(  \mathscr{H} ( \{ |z| \geq e^{1/N} \} ) , d_N \big)$, both equipped with the Skorokhod topology. Finally, let 
$ f_N :  D_N[0,\infty ) \rightarrow  D_{\mathscr{H}_N}[0,\infty ) $ be defined for $x \in D_N[0,\infty )$ by
\[ \begin{split}
	f_N(x) :\;   [0, &\infty ) \to \mathscr{H} ( \{ |z| \geq e^{1/N} \} )\\
	 &   t\quad   \mapsto \quad   H_N(x(t,\cdot )) \,.
\end{split} \]
	
We claim that the proof of Theorem \ref{Thol_new} amounts to showing that the map $f_N$ is continuous. Indeed, assume so. Then it follows from Theorem \ref{Tfields}, together with the continuous mapping theorem (cf. \cite{billingsley1999convergence}, Theorem 2.7), that 
$ f_N (\cF_n^{1/N} ) \to f_N (\cF^{1/N} )$ 
in distribution with respect to the Skorokhod metric on  $D_{\mathscr{H}_N} [0,\infty )$. 
Moreover, recall from \eqref{expand} the definition of $\cF$. Then, since by Lemma \ref{lemma_new} the stochastic process $\cF^{1/N} $ agrees in law with $\cF$ on $[0, \infty ) \times   e^{1/N}\bbT$, we deduce that the corresponding holomorphic extensions $f_N (\cF^{1/N} ) $ and $\cF$  agree in law on $[0,\infty ) \times \{ |z|\geq e^{1/N} \}$. It follows that  
$f_N (\cF_n^{1/N} ) \to \cF$ in distribution with respect to the Skorokhod metric on  $D_{\mathscr{H}_N} [0,\infty )$. Since $f_N (\cF_n^{1/N} ) \equiv \cF_n$ on $[0,\infty ) \times \{ |z| \geq e^{1/N} \}$, with $\cF_n$ defined as in \eqref{bigFn}, this shows that $\cF_n \to \cF$ in distribution with respect to the Skorokhod metric on  $D_{\mathscr{H}_N} [0,\infty )$. 
$N$ being arbitrary, we conclude that $\cF_n \to \cF$ in distribution with respect to the Skorokhod metric $\mathbbm{d}_\mathscr{H}$ on  $D_{\mathscr{H}} [0,\infty )$, which is what we wanted to show. \\

It remains to prove that for all $N \geq 1$ the map $f_N$ is continuous from $D_N[0,\infty )$ to $D_{\mathscr{H}_N}[0,\infty )$. 
Since $f_N$ only acts on the spatial component, it suffices to show that the holomorphic extension map $H_N$ is continuous from $\big( C(e^{1/N}\bbT), \| \cdot \|_\infty \big)$ to $\big( \mathscr{H}(\{ |z|\geq e^{1/N} \}) , d_N \big) $ 
for all $N\geq 1$. 
Indeed, take $(\phi_n )_n $, $\phi$ in $C(e^{1/N} \bbT )$ and suppose that $\| \phi_n - \phi \|_\infty \to 0$ as $n\to\infty$. Then we have: 
	\[ \begin{split} 
	d_N ( H_N(\phi_n) , H_N(\phi) ) & 
	= \sup_{|z| \geq e^{1/N}} \big| (H_N \phi_n)(z) - (H_N \phi )(z) \big| \wedge 1 
	= \sup_{|z| = e^{1/N}} \big| (H_N \phi_n)(z) - (H_N \phi )(z) \big| \wedge 1
	\\ & = \sup_{|z| = e^{1/N}} |  \phi_n(z) - \phi (z) |  \wedge 1
	\leq  \| \phi_n - \phi \|_\infty \to 0 \end{split}\]
as $n \to\infty$, where the second equality follows by applying the maximum principle, which is in force since $(H_N\phi_n) (z) \to 0$ as $|z|\to\infty$ by construction for all $N,n \geq 1$. This concludes the proof.
\end{proof}

\vspace{3mm}

\section{The boundary process} \label{explicit}
We have proved weak convergence to a limiting Gaussian process $\cF$ taking values in the space $\mathscr{H}$ of holomorphic functions on $\{ |z|>1\}$, of which we have provided an explicit construction. In this section we address the question of a rigorous definition of the boundary values $\cW$  of $\cF$, formally given by \eqref{Wformal}.

\subsection{Abstract Wiener Space construction}
It is clear that we have no hope to define $\cW (t, \cdot )$ pointwise, since the formal series \eqref{Wformal} diverges almost surely at each point. 
One could try, on the other hand, to make sense of it as a  complex Gaussian process taking values in a suitable Hilbert space of functions on the unit circle. 

For convenience of the reader, we review below the construction of Gaussian random variables, and then Gaussian stochastic processes, on infinite dimensional Hilbert spaces. This will then be applied to rigorously define the boundary process $\cW$. Our presentation follows \cite{da2014stochastic, sheffield2007gaussian}. 
 
 \subsubsection{Gaussian random variables on a Hilbert space}

\begin{Definition}\label{abstractWS} 
An Abstract Wiener Space is a triple $(H,B,\mu )$, where:
\begin{itemize}
\item[(i)]  $\big( H, (\cdot , \cdot )_H \big) $ is a Hilbert space,
\item[(ii)] $\big( B , \| \cdot \|_B \big)$ is the Banach space completion of $H$ with respect to the measurable norm $\| \cdot \|_B$ on $H$, equipped with the Borel $\s$-algebra $\cB$ induced by $\| \cdot \|_B$, and 
\item[(iii)] $\mu$ is the unique Borel probability measure on $(B, \cB )$ such that, if $B^*$ denotes the dual space of $B$, 
then $\mu \circ \phi^{-1} = \cN ( 0 , \| \tilde{\phi} \|_H^2 )$ for all $\phi \in B^*$, where $\tilde{\phi}$ is the unique element of $H$ such that $\phi ( h) = ( \tilde{\phi} , h )_H$ for all $h \in H$.  
\end{itemize} 
\end{Definition}

Note that   (iii) also reads as follows: if $X$ is a random variable on $(B,\cB)$ distributed according to $\mu$, then $\phi (X) \sim \cN(0 , \| \tilde{\phi} \|_H )$ for all $\phi \in B^*$. 
In this case we say\footnote{Note that, unless $\dim H <\infty$, a standard Gaussian random variable $X$ on $H$ does not take values in $H$, but only in the larger Banach space $B$.} that $X$ is a standard Gaussian random variable on $H$. 
We refer the reader to \cite{gross1967abstract, sheffield2007gaussian} for the definition of measurable norm on a Hilbert space, and for existence and uniqueness of such a measure $\mu$. Here we will only need the following two properties. 
\begin{Facts}
Let $H$ be a Hilbert space with inner product $(\cdot , \cdot )_H$, and let $\| \cdot \|_H^2 = (\cdot , \cdot )_H$. Then:
\begin{itemize}
\item[(a)] any measurable norm on $ H $ is weaker than $\| \cdot \|_H$, and
\item[(b)] if $T$ is a Hilbert-Schmidt operator on $H$, i.e. 
	\[ \sum_{i=1}^\infty \| T e_i \|_{H}^2 < \infty \, , 
	\qquad (e_i)_{i=1}^\infty \mbox{ ONB of } \big( H , ( \cdot , \cdot )_{H} \big) \, , \]
then  $\| T \cdot \|_{H} $ is a measurable norm on $H$. 
\end{itemize}
\end{Facts}
It is worth to point out that, unless $\dim (H)<\infty$, the norm $\| \cdot \|^2_H = ( \cdot , \cdot )_H$ is not measurable on $H$, so that $\| \cdot \|_B \neq \| \cdot \|_H$. When $\| \cdot \|_B$ is itself induced by an inner product, which will turn out to be the case in our construction, $B$ 
is itself a Hilbert space.  

\subsubsection{Brownian Motion on a Hilbert space}
\begin{Definition}
Let $(H,B,\mu )$ be an abstract Wiener space, and denote by $\mu_t$ the unique\footnote{Existence and uniqueness follow trivially by existence and uniqueness of $\mu$. } 
  probability measure on $B$, such that 
$\mu_t \circ \phi = \cN(0, t \| \tilde{\phi} \|_H^2)$ for all $\phi \in B^*$.
Finally, let $C_B[0,\infty )$ the space of continuous functions from $[0,\infty )$ to $B$, equipped with the $\s$-algebra generated by the coordinate functions $x \mapsto x(t)$. 
Then (cf. \cite{da2014stochastic}, pp.81-85) there exists a unique probability measure {\boldmath{$\mu$}} on $C_B[0,\infty )$ such that, if $W$ is a random variable in $C_B[0,\infty )$ distributed according to {\boldmath{$\mu$}}, then the following hold:
\begin{itemize}
\item $W(0)=0$ {\boldmath{$\mu$}}-a.s.
\item $W$ has independent increments,
\item for any $0\leq s < t$, $W(t) - W(s)$ is distributed according to $\mu_{t-s}$.
\end{itemize}
If a random variable $W$ on $C_B[0,\infty)$ is distributed according to {\boldmath{$\mu$}} we call it (cylindrical) Brownian Motion on $H$.
\end{Definition}
\begin{Proposition}[\cite{da2014stochastic}, Proposition 4.3] \label{Ptrivial}
Let $(H,B,\mu)$ be an Abstract Wiener Space, and let $(e_k)_{k} $ be an ONB of $H$ with respect to $(\cdot , \cdot )_H$. If $W$ is a Brownian Motion on $( H, (\cdot , \cdot )_H)$, then there exists a collection of i.i.d. real-valued Brownian Motions $(\beta_k)_k$ such that 
	\[ W(t) = \sum_{k} \beta_k (t) e_k \,,\]
where the above series converges in $L^2(B)$.
\end{Proposition}
Note that by setting $t=1$ in the above result we deduce the following.
\begin{Corollary}\label{Ctrivial}
Let $(H,B,\mu)$ be an Abstract Wiener Space. If $X$ is a standard Gaussian random variable on $H$, i.e. $X\sim \mu$, then there exists a collection $(A_k)_k$ of i.i.d. $\cN (0,1)$ real random variables such that $\ds  X = \sum_{k} A_k e_k $.
\end{Corollary}

\subsubsection{OU process on a Hilbert space}
Having constructed a Brownian Motion $W$ on $H$, one could then go ahead and define stochastic integration with respect to it. For a detailed account on the theory of stochastic integrals with respect to a Brownian motion taking values in an infinite-dimensional Hilbert space we refer the reader to \cite{da2014stochastic}, Chapter 4.
Here we are only interested in a very special case, namely the one of deterministic integrands which diagonalise on the ONB $(e_k)_{k}$ of $H$. Indeed, in this case the definition of stochastic integral with respect to $W$ reduces to the one of a countable collection of stochastic integrals on $\bbR$. 
\begin{Definition} \label{OUdef}
Let $(H,B,\mu)$ be an Abstract Wiener Space, and assume that the norm $\| \cdot \|_B$ on $B$ is induced by an inner product $(\cdot , \cdot )_B$, so that $B$ is itself a Hilbert space. Let  $(e_k)_{k}$ denote an ONB for $H$, and let $W(t) =\sum_k \beta_k(t) e_k$ be a Brownian Motion on $H$. Then, if  $\Psi (t) = \sum_{k} a_k(t) e_k$ for some collection of continuous real-valued functions $(a_k(t))_k $ such that $\sum_k \int_0^\infty |a_k(t)|^2 \dd t <\infty$, then 
	\[ \int_0^t \Psi (s) \dd W(s) := \sum_{k} \int_0^t a_k(s) \dd \b_k(s) e_k \, , 
	\qquad \forall t\geq 0 \, .\]
\end{Definition}

\vspace{1mm}

\subsection{The boundary process $\cW$} 
Let us now see how the above definitions read in the case of our interest.
Recall that  $D = \bbR / 2\pi \bbZ$, and for $f , g \in C^\infty (D)$ with Fourier expansion
$ f(\th ) =  \sum_{k\in \bbZ} \hat{f}_k e_k(\th)$, $ g(\th ) =  \sum_{k\in \bbZ} \hat{g}_k e_k(\th)$
introduce the inner product 
	\[ (f , g)_{H} := \sum_{k\neq 0}|k| \hat{f}_k \, \hat{g}_k \, . \]
Let $\sim$ be the equivalence relation on $C^\infty (D)$ which identifies two functions if they differ by a constant, and set $H_s (D) := C^\infty (D) / \sim $. We identify each equivalence class of $H_s(D)$ with its representative having zero average on $D$. Denote by $H$ the Hilbert space completion of $H_s (D)$ with respect to $(\cdot , \cdot )_{H}$. We seek to define an abstract Wiener space $(H , B , \mu)$ for a suitable choice of measurable norm  on $H$.

By property (b), in order to construct a measurable norm $\| \cdot \|_B $ on $H$ it suffices to find a Hilbert-Schmidt operator $T$ on $H$, and set $\| \cdot \|_B = \| T \cdot \|_{H} $.   In fact, we are going to construct a one-parameter family of such operators, and hence of measurable norms on $H$. 

For any $a \in \bbR$, let $(-\Delta )^a$ be the operator that acts on $L^2(D)$ functions by multiplying the Fourier coefficients by $ |k|^{2a}$, that is 
	\[ (-\Delta )^a \bigg(   \sum_{k\in \bbZ} \hat{f}_k e_k\bigg) (\th) := 
	 \sum_{k \neq 0} |k|^{2a} \hat{f}_k e_k (\th)  . \]
Set 
	\[ \cH_a = \Big\{ f \in L^2 (D) : (-\Delta )^a f \in L^2(D) \Big\} \diagup \sim \, , \]
and equip $\cH_a$ with the inner product $ ( f , g)_a := ( (-\Delta )^a f , (-\Delta )^a g )$. Then $\big( \cH_a , ( \cdot , \cdot )_a \big)$ is a Hilbert space, and in fact it is the Hilbert space completion of $H_s(D)$ with respect to the inner product $(\cdot , \cdot )_a$.
Denote by $\| \cdot \|_a$ the norm induced by $(\cdot , \cdot )_a$, and note that if $a>b$ then $\| \cdot \|_b \leq \| \cdot \|_a$, so that $\cH_a \subset \cH_b$. It follows that, whenever $b<a$, $\| \cdot \|_b$ is well defined on $\cH_a$, and moreover $\| f \|_b = \| (-\Delta )^{b-a} f \|_a$. In order for $\| \cdot \|_b$ to also be measurable on $\cH_a$, then, it suffices to show that $(-\Delta )^{b-a}$ is a Hilbert-Schmidt operator on $\cH_a$. To this end, note that if $f_k = (-\Delta )^{-a} e_k$ for $k\neq 0$, then $(f_k)_{k\neq 0}$ is an ONB for $\cH_a$. Moreover,
	\[ \sum_{k\neq 0} \|(-\Delta )^{b-a} f_k \|_a^2 = \sum_{k\neq 0} \big( (-\Delta )^b f_k , (-\Delta )^b f_k \big) = \sum_{k\neq 0} \big( (-\Delta )^{b-a} e_k , (-\Delta )^{b-a} e_k \big) = 2 \sum_{k\geq 1} k^{4(b-a)} \, , \]
which converges if and only if $b<a-1/4$. We have therefore proved the following.

\begin{Proposition} \label{PMeas}
Whenever $ b < a - 1/4$ the norm $\| \cdot \|_b$ is  measurable on $\cH_a$. 
\end{Proposition}

Note that $ ( f , g)_{H} = \big( (-\Delta )^{1/4} f , (-\Delta )^{1/4} g \big) $ by definition, so $(-\Delta )^{-1/4}$ provides a Hilbert space isomorphism between the spaces $H$ and $\cH_{1/4}$, that when convenient  we identify. It follows that in order to get a measurable norm on $H$ it suffices to have a measurable norm on $\cH_{1/4}$. 
By Proposition \ref{PMeas}, any norm of the form $\| \cdot \|_{-\e}$ for $\e >0$ will do. 

Let $\cB_a$ denote the Borel $\s$-algebra on $\cH_a$ for the norm $\| \cdot \|_a$. Then for any $\e >0$ there exists a unique Borel probability measure $\mu_{-\e}$ on $(\cH_{-\e} , \cB_{-\e})$ such that (iii) in Definition \ref{abstractWS} holds, i.e. $\mu_{-\e}  \circ \phi^{-1} = \cN ( 0 , \| \tilde{\phi} \|_{H} )$ for all $\phi$ continuous linear functionals on $\cH_{-\e}$,  
where $\tilde{\phi}$ is the unique element of $H$ such that $\phi ( h) = ( \tilde{\phi} , h )_{H}$ for all $h \in H$.  It follows that $(H, \cH_{-\e} , \mu_{-\e} )$ is an Abstract Wiener Space for any $\e >0$. 

\begin{Definition}
A random variable $X$ is said to be a standard Gaussian on the Hilbert space $H \cong \cH_{-1/4}$ if $X \sim \mu_{-\e}$ as random variable on $\cH_{-\e}$, for all $\e >0$.
\end{Definition}

\begin{Remark}
This definition is consistent, meaning that if $-\e_1 < -\e_2$, so that $\cH_{-\e_2} \subset \cH_{-\e_1}$, then by property (iii) of Definition \ref{abstractWS} the restriction of $\mu_{-\e_1}$ to $\cH_{-\e_2}$ coincides with $\mu_{-\e_2}$.  
\end{Remark}

To conclude this abstract construction, we point out that $X$ as above has a very simple expansion in Fourier basis. Indeed, if $f_k = (-\Delta )^{1/4} e_k = e_k / \sqrt{|k|}$, then $(f_k)_k$ is an ONB for $\cH_{-1/4} \cong H$. It then follows from Corollary \ref{Ctrivial} that there exists a collection $(A_k)_{k\in\bbZ}$ i.i.d. real $\, \cN (0,1)$ random variables such that
 	\[ X(\th ) = \sum_{k\neq 0}  A_k f_k(\th) = 
 	\sum_{k\neq 0} \frac{A_k}{\sqrt{|k|}}  \, e_k(\th) \,.\]

Moreover, in light of Proposition \ref{Ptrivial} it is now easy to  construct a Brownian motion $W$ on $H$: simply take a collection of i.i.d. real-valued Brownian Motions $(\beta_k)_{k\in \bbZ}$, and set 
	\[ W(t,\th) = \sum_{k\neq 0} \beta_k (t) f_k(\th ) 
	= \sum_{k\neq 0} \frac{\beta_k(t)}{\sqrt{|k|}}  \, e_k(\th) \,,\]
where the above series converge in $L^2(\cH_{-\e})$, for any $\e >0$.
From $W$ we construct an Ornstein-Uhlenbeck process on $H$ as follows. Let $\Psi : [0,\infty ) \to B$ be defined by 
	\[ \Psi (t) = \sum_{k\neq 0} \sqrt{2|k|} e^{-|k|t} f_k 
	= \sum_{k\neq 0} e^{-|k|t} e_k \, , \]
and set $  \tilde{\cW} (t) := \int_0^t \Psi (t-s) \dd W(s) $ for all $t\geq 0$.
Then by Definition \ref{OUdef} we have
	\[ \tilde{\cW} (t,\th) = \sum_{k\neq 0} \bigg[ 
	\int_0^t  e^{-|k|(t-s)} \dd \beta_k(s) \bigg] e_k (\th )\, , \]
so that the coefficients of $\tilde{\cW}$ perform i.i.d. OU processes on $\bbR$. 
Take an independent copy $\tilde{\cW}'$ of $\tilde{\cW}$. Then we have 
	\[  \frac{\tilde{\cW} (t , \cdot )  + i\tilde{\cW}' (t , \cdot )  }{\sqrt{2}} = 
	\sum_{k\neq 0} \Big( \frac{A_k(t) + iB_k(t) }{\sqrt{2}} \Big) e_k ( \cdot ) \, , \]
for i.i.d. OU processes $(A_k)_k$, $(B_k)_k$ on $\bbR$, which equals the r.h.s. of \eqref{Wformal}. This provides a rigorous construction of the boundary process $\cW$ of the limiting holomorphic field $\cF$ as an OU process on $\cH_{-\e}$, for any $\e >0$. 

\begin{Remark}
Note that, by linearity, 
	\[\begin{split} 
	 \dd \cW (t, \cdot ) & =  - \Big[ \sum_{k\neq 0} |k| \Big( \frac{A_k(t) + iB_k(t) }{\sqrt{2}} \Big) e_k (\cdot )\Big] \dd t + \sqrt{2} \Big[ \sum_{k\neq 0} \frac{\dd \beta_k (t) + i \dd \beta_k ' (t) }{\sqrt{2}} e_k (\cdot ) \Big] \, , 
	 \end{split} \]
where we have denoted by $(\b_k )$, $(\b_k ')$ the i.i.d. real Brownian Motions driving the OU processes $(A_k)$, $(B_k)$. This
 shows that $\cW$ solves the Fractional Stochastic Heat Equation 
	\[  \dd \cW (t, \cdot ) = -(-\Delta )^{1/2} \cW (t, \cdot ) + \sqrt{2} \, \dd \xi (t, \cdot ) \]
for $\xi $ complex space-time white noise on the unit circle $\bbT$. 
\end{Remark}

\begin{Remark}
We point out that the convergence result $\cF_n \to \cF$ of Theorem \ref{Thol_new} can be interpreted as a convergence result for the corresponding boundary values, seen as distributions acting on a suitable space of test functions. More precisely, let $\cW_n$ denote the boundary values of $\cF_n$, so that $\cW_n (t, \th ) = \cF_n (t,e^{i\th} )$ for $(t,\th ) \in [0,\infty ) \times [-\pi , \pi )$. As space of test functions we take 
	\[ \big\{ \varphi \in C^\infty (\bbT ) : \varphi = P_r * \psi \mbox{ for some } \psi \in C^\infty (\bbT ) \mbox{ and some }r>1 \big\} \, . \] 
For each such $\varphi$, we have:
	\[ ( \cW_n(t, \cdot ) , \varphi ) = ( \cW_n (t, \cdot ) , P_r * \psi ) = ( P_r * \cW_n (t,\cdot )  , \psi ) = (\cF_n (t, re^{i\cdot} ) , \psi ) \to ( \cF (t, r e^{i \cdot} ) , \psi ) \]
as $n \to \infty$ in distribution, as continuous functions on $\bbT$. It would be interesting to understand if such convergence holds for a larger class of test functions, and ultimately as stochastic processes taking values in $\cH_{-\e}$ for $\e >0$.
\end{Remark}

\vspace{3mm}

\section{Further developments} \label{Sopen}
As mentioned in the introduction, arguably  the main open problem in this area is to obtain rigorous results on (the non-regularised version of) HL($\a$) models with $\a >0$, with a particular interest for $\a \in [1,2]$. 
In this last section, on the other hand, we would like to collect some open questions for the $\a =0$ case, which has proven to be already  very interesting from a mathematical point of view, and it has the advantage of being  more tractable than the case $\a >0$ due to its intrinsic i.i.d. structure. 

\subsection*{Maxima of the cluster boundary}
Once the question of fluctuations of cluster boundary has been settled, one is led to wonder about the asymptotic behaviour of the maxima of the cluster. For $n\geq 0$, set $m_n = \sup \{ |z| : z \in K_n \}$. It follows from the analysis carried out in \cite{norris2012hastings} that $| m_n -e^{cn} | \to 0$ almost surely as $n\to\infty$, $ c\to 0$ and $nc \to t$ for some $t>0$. What about fluctuations around this deterministic behaviour?

\subsection*{Connections with IDLA}
Let us recall the definition of a different, discrete model of aggregation on the plane. Consider the lattice $\bbZ^2$, set $K_0=\{(0,0)\}$ to be the initial cluster, and  define $K_1 \subset K_2 \subset K_3 \ldots $ recursively as follows. At each step $n\geq 1$, start a random walk $X_n = (X_n(k))_{k\in \bbN}$ (independent of everything else) from the origin, and, if $T_n = \inf \{ k\geq 0 : X_n(k) \notin K_{n-1} \}$ denotes the first time $X_n$ exits the cluster $K_{n-1}$, set $K_n = K_{n-1} \cup \{ X_n (T_n ) \}$. This grows an increasing family of connected clusters $(K_n)_{n\geq 0} \subset \bbZ^2$, whose dynamics is usually referred to as \emph{Internal Diffusion Limited Aggregation} (in short IDLA) on  $\bbZ^2$. IDLA was originally introduced by Meakin and Deutch in \cite{meakin1986formation}, and its large $n$ behaviour is by now well understood in any dimension. 
We focus here on the $2$-dimensional case, which we have described above. It is shown in \cite{lawler1992internal} that $K_n \approx B_{r(n)} (0)$ for $n$ large enough, where $B_r(0)$ denotes the Euclidean ball of radius $r$ centred at the origin, and $r(n) = \sqrt{n/\pi }$, so that $\pi r^2(n) = n$. In \cite{asselah2013logarithmic,jerison2012logarithmic} fluctuations around this asymptotic spherical shape are analysed. In particular, the authors show that the maximum fluctuations of $K_n$ are $\mathcal{O} (\log n )$ for large $n$. 
The average fluctuations are then studied in \cite{jerison2014internal}, and are shown to be given by the restriction of a variant of the $2$-dimensional Gaussian Free Field, so called \emph{Augmented GFF}, to the unit circle. If, on the other hand, one considers IDLA on the $2$-dimensional cylinder $(\bbZ / N \bbZ ) \times \bbZ_+$  rather than the whole plane, then (cf. \cite{jerison2014internal2}) the fluctuation field coincides \emph{exactly} with the restriction of the whole plane GFF to the unit circle, i.e. our limiting fluctuation field $\cW_\infty$. One is then led to wonder whether the connection between HL($0$) and IDLA clusters goes beyond having the same scaling limit and fluctuations. 

\subsection*{Log-correlation and branching structure}
Log-correlated Gaussian fields appear to arise in correspondence with underlying branching structures. To make this statement more precise, let us give some examples.
\subsubsection*{Branching Random Walk}
For $N \in \bbN_+$ consider a binary tree $G_N = (V_N , E_N)$ of depth $N$, and assign to each edge $e \in E_N$ a standard Gaussian random variable $X_e$ in an independent fashion. Let $L_N$ denote the set of leaves of the binary tree (so that $|L_N|=2^N$), and  for each $v \in L_N$ denote by  $\g_v$ the unique path from $v$ to the root. Then set
	\[ Y_v = \sum_{e \in \g_v} X_e  \, . \]
It is easy to see that, when the set of leaves is seen as embedded in the unit interval $[0,1]$ (i.e. the leaves are identified with $2^N$ equispaced points in $[0,1]$), the correlation between two leaves $v,v' \in L_N$ decays logarithmically with their Euclidean distance $|v-v'|$, for all $N$. In other words, at each level $N\geq 1$ the random field $(Y_v)_{v\in L_N}$ is a log-correlated (Gaussian) field.

\subsubsection*{Branching in IDLA}
A branching structure is also present in the IDLA model described above. Indeed, one could imagine to colour, each time a new particle is added to the cluster, the last edge along which the correspondent random walk has jumped in order to exit the cluster. More precisely, define the set of coloured edges recursively as follows. Given the cluster $K_{n-1}$ with $n-1$ particle, start an independent random walk $X_{n}$ from the origin and, if $T_{n} $ is the first exit time of $X_n$ from $K_{n-1}$, colour the edge $(X_n (T_{n} -1) , X_n (T_n ) )$. It is easy to see that the set of coloured edges forms a tree. Moreover, if one considers two points on the cluster boundary (which is asymptotic circular), then it follows from the work in \cite{jerison2014internal} that the correlation between fluctuations from circularity at the two points decays logarithmically with their distance. 

\subsubsection*{Branching in HL growth}
The fact that a branching mechanism arises in the growth of HL($0$) clusters was discovered by J. Norris and A. Turner in \cite{norris2012hastings}. 
Think for simplicity of the particle $P$ as being a slit, i.e.  $P=[1,1+\d ] \subset \bbC$. Then each time a new particle arrives it gets attached to exactly one particle in the cluster, at exactly one point, almost surely. Call this \emph{parent particle}. This identifies an ancestral structure in the cluster. Given any finite set of points, it is shown in \cite{norris2012hastings} that the corresponding ancestral lines converge to backwards coalescing Brownian Motions, i.e. Brownian Motions which evolve independently backwards in time until they coalesce, at which point they merge.
This result can be rephrased by saying that the cluster boundary converges, in a finite-dimensional sense, to the so called \emph{Brownian Web}.  
Here the underlying (forward) branching structure is then really a backwards  coalescing structure of the cluster boundary\footnote{Note that the dynamics of backwards coalescing Brownian Motions is different from the one of forward branching Brownian Motions, since for example in the latter case paths can intersect after branching.}, but nonetheless one obtains log-correlated fluctuations. \\

In all the above examples log-correlated Gaussian fields arise as fluctuation fields of several different models, all featuring some kind of underlying branching structure. We believe it would be very interesting to understand to what extent this is a general phenomenon, and, if so, how robust it is with respect to variations of the branching mechanism (e.g. introduction of correlation between branches).


\appendix

\section{Local fluctuations} \label{Aloc}
Our main result on local fluctuations is the following. 
\begin{Theorem}\label{TlocBig}
Pick any $t>0$, and let $z = e^{ia+\s}$, $w=e^\s$ for some $a \in [-\pi , \pi )$, $\s >0$. Define $\e = \d^{2/3} \log (1/\d )$ as in Theorem \ref{teoJ}, and assume that $\s \to 0$ as $c\to 0$, with $\s \gg \sqrt{\e}$. 
Then,  as $n \to \infty$, $nc \to t$ and $\s \to 0$,  it holds:
	\[\left(  \frac{\log \frac{\Phi_n (z)}{z} -nc }{\sqrt{c \log (\frac{1}{2\s})}} \,  , \,  
	\frac{\log \frac{\Phi_n (w)}{w} -nc }{\sqrt{c \log (\frac{1}{2\s})}}  \right)
	\longrightarrow \, 
	(\cN_1 , \cN_2 )  \]
in distribution, where $(\cN_1 ,\cN_2)$ is a random vector with centred complex Gaussian entries, and covariance structure given by 
	\[ \bbE 
    (\cN_1 \cN_2 ) = \left( \begin{matrix}
	\frac{1}{1+\a^2} & -\frac{\a }{1+\a^2} \\
	\frac{\a }{1+\a^2} & \frac{1}{1+\a^2} 
	\end{matrix} \right) \, , 
	\]
for $ \a = \lim_{\s \to 0} \frac{a}{2\s} \in [0, \infty ]$ (with the convention that $\frac{1}{1+\a^2} = \frac{\a}{1+\a^2} =0$ when $\a =\infty$). 
\end{Theorem}
Theorem \ref{TlocBig} follows by the same arguments that lead to Theorem \ref{Tfields}, considering now $\frac{\cX_{k,n}^\s ( \cdot )}{\sqrt{ \log ( \frac{1}{2\s} ) } }$ in place of $\cX_{k,n}^\s (\cdot )$. Under the assumption $\s \gg \sqrt{\e}$ Lemmas \ref{le3}-\ref{replace} still apply, so that Theorem \ref{ML_CLT} holds. Together with Proposition \ref{simpler}, this shows that in the limit as $\s \to 0$, $n \to\infty$ and $nc \to t$, real and imaginary part of $\frac{\log \frac{\Phi_n (e^{ia+\s})}{e^{ia+\s}} -nc }{\sqrt{c \log (\frac{1}{2\s})}} $ are asymptotically i.i.d. centred Gaussians, with limiting variance given by 
	\[ \lim_{\s \to 0} \frac{1}{\log \big( \frac{1}{2\s} \big)} \bigg[ 
    \frac{1}{2\pi} \int_{\s}^{\s+t} \int_{-\pi}^{\pi} \bigg[ \mathrm{Re} \bigg( \frac{e^{-i\th + x} +1}{e^{-i\th + x} -1}\bigg) \bigg]^2 \dd \th \dd x - t \bigg] 
    = \lim_{\s \to 0}  \frac{1}{\log \big( \frac{1}{2\s} \big)}  
    \log \bigg| \frac{1-e^{-2(\s +t)}}{1-e^{-2\s}} \bigg|  = 1 \, . \]
For two point correlation we reason as in Section \ref{Sfdd}, to gather that the limiting covariance between $\mathrm{Re} \Big(  \frac{\log \frac{\Phi_n (e^{ia+\s})}{e^{ia+\s}} -nc }{\sqrt{c \log (\frac{1}{2\s})}} \Big)$ and $\mathrm{Re} \Big( \frac{\log \frac{\Phi_n (e^{\s})}{e^{\s}} -nc }{\sqrt{c \log (\frac{1}{2\s})}} \Big)$ is given by 
	\[ \begin{split} 
	\lim_{\s \to 0} \,  \frac{1}{\log \big( \frac{1}{2\s} \big)}  & \bigg[ 
    \frac{1}{2\pi} \int_{\s}^{\s+t} \int_{-\pi}^{\pi}  \mathrm{Re}  \bigg( \frac{e^{-i(\th -a) + x} +1}{e^{-i(\th -a) + x} -1}  \bigg) 
    \mathrm{Re} \bigg( \frac{e^{-i\th + x} +1}{e^{-i\th + x} -1}\bigg)  \dd \th \dd x - t \bigg] =
    \\ & 
     = \lim_{\s \to 0}  \frac{1}{\log \big( \frac{1}{2\s} \big)}  
    \log \bigg| \frac{1-e^{-2(\s +t)+ia}}{1-e^{-2\s+ia}} \bigg|  
    = \lim_{\s \to 0} \frac{ \log ( 1 + e^{-4\s} -2e^{-2\s} \cos a )}{2 \log 2\s} = \frac{1}{1+\a^2} 
    \end{split}\]
whenever $a/2\s \to \a \in [0,\infty ]$, where the last equality is obtained by Taylor expanding around $a=0 , \s =0$.  The same holds for imaginary parts correlation.

Finally, the asymptotic covariance between $\mathrm{Re} \Big(  \frac{\log \frac{\Phi_n (e^{ia+\s})}{e^{ia+\s}} -nc }{\sqrt{c \log (\frac{1}{2\s})}} \Big)$ and $\mathrm{Im} \Big( \frac{\log \frac{\Phi_n (e^{\s})}{e^{\s}} -nc }{\sqrt{c \log (\frac{1}{2\s})}} \Big)$ is given by 
	\[ \begin{split} 
	\lim_{\s \to 0} \,  \frac{1}{\log \big( \frac{1}{2\s} \big)}  & \bigg[ 
    \frac{1}{2\pi} \int_{\s}^{\s+t} \int_{-\pi}^{\pi}  \mathrm{Re}  \bigg( \frac{e^{-i(\th -a) + x} +1}{e^{-i(\th -a) + x} -1}  \bigg) 
    \mathrm{Im} \bigg( \frac{e^{-i\th + x} +1}{e^{-i\th + x} -1}\bigg)  \dd \th \dd x - t \bigg] =
    \\ & 
     = \lim_{\s \to 0}  \frac{1}{\log \big( \frac{1}{2\s} \big)}  
    \mathrm{Arg} \bigg( \frac{1-e^{-2(\s +t)+ia}}{1-e^{-2\s+ia}} \bigg) 
    = \lim_{\s \to 0}  \frac{1}{\log2\s} \arctan \bigg( \frac{\sin a }{\cos a - e^{2\s}} \bigg)   
    = \frac{\a}{1+\a^2} 
    \end{split}\]
whenever $a/2\s \to \a \in [0,\infty ]$, where the last equality is obtained by Taylor expanding around $a=0 , \s =0$. This concludes the proof of Theorem \ref{TlocBig}.

\vspace{5mm}

\section{Proof of Lemma \ref{Llong}} \label{Alemma}
Fix any $k\leq \ns$ and set for simplicity
$ Z_1 = Z_{k,\nt} (a) , \quad Z_2 = Z_{k,\ns} (0) , \quad 
	W_1 = e^{ia + \s + (\nt -k)c} $, $ W_2 = e^{\s + (\ns -k)c }$.
Moreover, introduce the functions $g_\th ( z) = \mathrm{Re} \big( \log \frac{F(e^{-i\th} z)}{e^{-i\th} z} \big) $, $h_\th ( z) = c \,\mathrm{Re} \big(\frac{e^{-i\th }z +1}{e^{-i\th}z-1} \big)$. 
Then $\bbE(X_{k,\nt }(a)X_{k,\ns}(0)| \mathscr{F}_{k+1,\nt }) = \frac{1}{2\pi c } \int_{-\pi}^\pi g_\th ( Z_1) g_\th ( Z_2 ) \dd \th -c$, and we have to show that 
	\begin{equation}\label{goal}
    \bigg| \frac{1}{2\pi c } \int_{-\pi}^\pi g_\th ( Z_1) g_\th ( Z_2 ) \dd \th - 
    \frac{1}{2\pi c} \int_{-\pi}^\pi h_\th ( W_1 ) h_\th ( W_2 ) \dd \th \bigg| 
    \leq \frac{C(t) c \e}{(\s+(\ns -k)c)^3} \, . 
	\end{equation}
Trivially, the l.h.s. is bounded above by 
	\[
    \frac{1}{2\pi c} \int_{-\pi}^\pi \Big( |g_\th ( Z_2 )| \cdot | g_\th ( Z_1) - h_\th ( W_1)| 
    + |h_\th ( W_1)| \cdot | g_\th ( Z_2 ) - h_\th ( W_2 )| \Big) \dd \th \, . \]
We bound each term separately. Recall the definition of  $E(m,\e )$, from which it follows that 
	\[ \begin{split} 
	& \max_{\th \in [-\pi , \pi )} \Big\{ |Z_{k,\nt} (\th )| \vee  |Z_{k,\ns} (\th )| \Big\}
	\leq e^{\s + (\nt -k)c} (1+2\e ) \leq 2e^{t+2} = C(t) \, , \\
	& \min_{\th \in [-\pi , \pi )} \Big\{ |Z_{k,\nt} (\th )| \wedge |Z_{k,\ns} (\th )|  -1 \Big\}  \geq e^{\s + (\ns-k)c} (1-2\e ) 
	\geq \frac{\s + (\ns -k)c }{2} 
	\end{split} \]
as long as 
$n$ is large enough and $\s \gg \e$. We combine the above estimates with the bounds in Corollary \ref{Fcor}, to get
	\begin{equation}\label{p1}
	\begin{split} 
	 | g_\th (Z_2) | & \leq \bigg| \log \frac{F(e^{-i\th} Z_2 )}{e^{-i\th} Z_2 } - 
	c \frac{e^{i\th} Z_2 +1}{e^{-i\th} Z_2 -1} \bigg| + c \bigg| \frac{e^{i\th} Z_2 +1}{e^{-i\th} Z_2 -1} \bigg|  \leq \frac{C c^{3/2} |Z_2|^2}{(|Z_2|-1)^3} + \frac{2c |Z_2|}{|Z_2|-1} \\
	& \leq \frac{C(t)c^{3/2}}{(\s + (\ns -k)c)^3} + \frac{C(t)c}{\s + (\ns -k)c} 
	\leq \frac{2C(t)c}{\s + (\ns -k)c} 
	\end{split} 
	\end{equation}
for $n$ large enough. Similarly, we find
	\begin{equation}\label{p2}
	 |h_\th (W_1 )| \leq c \bigg( 1+\frac{2}{|W_1|-1}\bigg) \leq c \bigg( 1 + \frac{2}{\s + (\ns -k)c} \bigg) \leq \frac{C(t)c}{\s + (\ns -k)c} 
	 \end{equation}
on the event $E(m,\e )$ for, say, $C(t) = t+3$ and $n$ large enough.

In order to bound the remaining terms, observe that, by definition, on $E(m,\e )$ we have 
	\[ \max_{\th \in [-\pi , \pi )} \Big\{ \big| Z_{k,\ns } (\th ) - e^{i\th + \s + (\ns -k)c} \big| \vee  \big| Z_{k,\nt } (\th ) - e^{i\th + \s + (\nt -k)c} \big| \Big\}
	\leq C(t) \, \e \, , \]
from which we get $|Z_1 - W_1| \leq C(t) \e $ and $|Z_2 - W_2 | \leq C(t) \e$. Combining this with Corollary \ref{Fcor} we finally obtain 
	\[ \begin{split} 
	| g(Z_1) - h(W_1)| & \leq |g(Z_1) - h(Z_1)| + |h(Z_1) - h(W_1)| \leq 
	\frac{C c^{3/2} |Z_1|^2}{(\s + (\ns -k)c)^3} + 
	\frac{2c |Z_1 - W_1|}{(|Z_1|-1)(|W_1|-1)} 
	\\ & \leq \frac{C(t)c^{3/2}}{(\s + (\ns-k)c)^3} + \frac{C(t)c\e }{(\s + (\ns-k)c)^2} 
	\leq \frac{2C(t)c\e}{(\s + (\ns-k)c)^2} 
	\end{split} \]
for $n$ large enough. Similarly one shows that the same bound holds for $| g(Z_2) - h(W_2)|$.
Putting now this together with  \eqref{p1} and \eqref{p2} gives \eqref{goal}. 

The second statement of the lemma follows by the same arguments, and the proof is omitted.


\bibliography{HLbib2}{}
\bibliographystyle{hplain}
\vspace{5mm}

\end{document}